\documentclass[11pt]{article}
\usepackage[usenames]{color} 
\usepackage{amssymb} 
\usepackage{amsmath} 
\usepackage[utf8]{inputenc} 
\usepackage{graphicx}
\usepackage{color,amsmath, amsthm, amsfonts, amssymb, dsfont, graphicx, listings, graphics, epsfig, enumerate, bbm,dsfont,mathrsfs}

\usepackage[margin=1in]{geometry}
\usepackage[fleqn,tbtags]{mathtools}
\usepackage[english]{babel}

\makeatletter \renewenvironment{proof}[1][\proofname] {\par\pushQED{\qed}\normalfont\topsep6\p@\@plus6\p@\relax\trivlist\item[\hskip\labelsep\bfseries#1\@addpunct{.}]\ignorespaces}{\popQED\endtrivlist\@endpefalse} \makeatother
\newtheorem{theorem}{Theorem}[section]
\newtheorem{corollary}[theorem]{Corollary}
\newtheorem{lemma}[theorem]{Lemma}
\newtheorem{proposition}[theorem]{Proposition}

\newtheorem{definition}[theorem]{Definition}
\theoremstyle{definition}
\newtheorem{example}[theorem]{Example}
\newtheorem{remark}[theorem]{Remark}

\numberwithin{equation}{section}
\newcommand{\R}{\mathbb{R}}
\newcommand{\N}{\mathbb{N}}

\DeclareMathOperator{\gr}{Graph}
\DeclareMathOperator{\dom}{dom}

\DeclareMathOperator{\cl}{cl}
\DeclareMathOperator{\intt}{int}

\DeclareMathOperator{\conv}{conv}
\DeclareMathOperator{\cconv}{\overline{conv}}
\DeclareMathOperator{\ch}{ch}
\newcommand{\pchd}{\operatorname{chd}_p}
\newcommand{\schd}{\operatorname{chd}_s}
\newcommand{\ochd}{\operatorname{chd}_0}
\newcommand{\pchcd}{\operatorname{chcd}_p}
\newcommand{\schcd}{\operatorname{chcd}_s}
\newcommand{\ochcd}{\operatorname{chcd}_0}

\DeclareMathOperator{\cch}{\overline{ch}}
\DeclareMathOperator{\dec}{dec}
\newcommand{\clpdec}{\operatorname{cl}_p\,\operatorname{dec}}
\newcommand{\clsdec}{\operatorname{cl}_s\,\operatorname{dec}}
\newcommand{\clodec}{\operatorname{cl}_0\,\operatorname{dec}}
\newcommand{\convdec}{\operatorname{conv}\,\operatorname{dec}}

\DeclareMathOperator{\cconvdec}{\overline{conv} \, dec}
\DeclareMathOperator{\cdecconv}{\overline{dec} \, conv}
\newcommand{\clp}{\operatorname{cl}_p}
\newcommand{\clw}{\operatorname{cl}_{\mathrm{w}}}

\let\abs=\envert

\newcommand{\Id}{\text{Id}}

\newcommand{\C}{\mathcal{C}}
\renewcommand{\O}{\Omega}
\renewcommand{\o}{\omega}

\newcommand{\F}{\mathcal{F}}
\newcommand{\Fhat}{\hat{\mathcal{F}}}
\newcommand{\T}{\mathcal{T}}
\newcommand{\B}{\mathcal{B}}
\newcommand{\Leb}{\mathcal{L}}
\newcommand{\D}{\mathscr{D}}
\newcommand{\K}{\mathcal{K}}

\newcommand{\A}{\mathcal{A}}
\newcommand{\E}{\mathbb{E}}

\newcommand{\M}{\mathcal{M}}

\newcommand{\Kp}{\mathcal{K}_p}
\newcommand{\p}{\mathbb{P}}

\renewcommand{\a}{\alpha}
\renewcommand{\b}{\beta}

\renewcommand{\d}{\delta}
\renewcommand{\D}{\Delta}
\renewcommand{\l}{\lambda}

\newcommand{\eps}{\varepsilon}

\newcommand{\calLov}{\mathfrak{L}^0(E)}
\newcommand{\chdLov}{\mathfrak{L} _{chd} ^0(E)}
\newcommand{\chcdLov}{\mathfrak{L} _{chcd} ^0(E)}
\newcommand{\Lpv}{L^p(E)}
\newcommand{\calLpv}{\mathfrak{L}^p(E)}
\newcommand{\chdLpv}{\mathfrak{L} _{chd} ^p(E)}
\newcommand{\chcdLpv}{\mathfrak{L} _{chcd} ^p(E)}

\newcommand{\calLsv}{\mathfrak{L}^s(E)}
\newcommand{\chdLsv}{\mathfrak{L} _{chd} ^s(E)}
\newcommand{\chcdLsv}{\mathfrak{L} _{chcd} ^s(E)}

\newcommand{\of}[1]{\ensuremath{\left( #1 \right)}}
\newcommand{\cb}[1]{\ensuremath{ \left\{ #1 \right\} }}

\newcommand{\norm}[1]{\ensuremath{ \left\Vert #1 \right\Vert }}
\newcommand{\ip}[1]{\ensuremath{ \left\langle #1 \right\rangle }}

\def\prehp(#1,#2){\ensuremath{  #1 \cdot #2 }}

\title{Random sets and Choquet-type representations}
\author{\c{C}a\u{g}{\i}n Ararat\thanks{Bilkent University, Department of Industrial Engineering, Ankara, Turkey, cararat@bilkent.edu.tr}\and {Umur Cetin\thanks{Cornell University, School of Operations Research and Information Engineering, Ithaca, NY, USA, sc2955@cornell.edu.}}
}
\date{January 16, 2022}

\begin{document}
\maketitle

\begin{abstract}
As appropriate generalizations of convex combinations with uncountably many terms, we introduce the so-called Choquet combinations, Choquet decompositions and Choquet convex decompositions, as well as their corresponding hull operators acting on the power sets of Lebesgue-Bochner spaces. We show that Choquet hull coincides with convex hull in the finite-dimensional setting, yet Choquet hull tends to be larger in infinite dimensions. We also provide a quantitative characterization of Choquet hull. Furthermore, we show that Choquet decomposable hull of a set coincides with its (strongly) closed decomposable hull and the Choquet convex decomposable hull of a set coincides with its Choquet decomposable hull of the convex hull. It turns out that the collection of all measurable selections of a closed-valued multifunction is Choquet decomposable and those of a closed convex-valued multifunction is Choquet convex decomposable. Finally, we investigate the operator-type features of Choquet decomposable and Choquet convex decomposable hull operators when applied in succession. \\ 
\\[-5pt]
\\[-5pt]
\end{abstract}

\section{Introduction}\label{sec:intro}

Integration of multifunctions dates back to Aumann~\cite{aumann1965}. The Aumann integral $\int X d\p$ of an $\F$-measurable multifunction $X \colon \O \to \mathcal{P}(\R^d)$ is defined as the collection of expectations of integrable random vectors taking values in $X$, provided that at least one such random vector exists. Here, $(\O,\F,\p)$ is a probability space and $\mathcal{P}(\R^d)$ denotes the power set of $\R^d$.

For a measurable multifunction $X \colon \O \to \mathcal{P}(\R^d)$ satisfying some mild boundedness conditions, one has 
\[
 \int \conv X\ d\p=\conv \Big( \int X d\p \Big),
\]
which is known as the Aumann identity; see Theorems 1.2.26 and 1.2.27 in \cite{molchanov2005}. The Aumann identity shows the importance of having a convex representation $Y = \conv X $ for a multifunction $Y$ in terms of another multifunction $X.$ 


A probability measure $\mu$ on $\R^d$ is said to have a \emph{barycenter} if there is a point $y \in \R^d$ satisfying 
\[
\ip{u,y}= \int_{\R^d} \ip{u,x} \, \mu(dx)
\]
for every $u \in \R^d.$ Notice that a convex combination $\sum_{i=1}^m \lambda_i x_i$ of points in $\R^d$ can be seen as the barycenter $\int_{\R^d} x \, \mu(dx)$ of $\mu = \sum_{i=1}^m \lambda_i \delta_{x_i}$ since for every $u \in \R^d ,$ one has
\begin{align*}
\ip{u, \int_{\R^d} \negthinspace x \mu(dx) } &=\int_{\R^d} \negthinspace \ip{u,x} \mu(dx) =\int_{\R^d} \negthinspace \ip{u,x}   \sum_{i=1}^m \lambda_i \delta_{x_i}(dx) \\
&=\sum_{i=1}^m \lambda_i  \int_{\R^d}\negthinspace \ip{u,x} \delta_{x_i}(dx)   = \sum_{i=1}^m \lambda_i \ip{u, x_i}   = \ip{u,\sum_{i=1}^m \lambda_i x_i} .
\end{align*}

In this paper, by generalizing convex combinations, we introduce \emph{Choquet combinations} as barycenters of probability measures. Then, we call a set a \emph{Choquet set} if it contains every Choquet combination of its points and define the \emph{Choquet hull} of a set as the smallest Choquet set containing it. We show that convex hull and Choquet hull coincide for every subset of $\R^d.$ However, in infinite-dimensional spaces, there are convex sets that are not Choquet sets, hence Choquet hulls tend to be strictly larger than convex hulls, yet we show that every closed convex set is a Choquet set. We also show that the Choquet hull of a set precisely consists of the collection of all Choquet combinations of its points, that is, every Choquet combination of Choquet combinations of points of a set can indeed be seen as a Choquet combination of its points.

In this spirit, we consider Choquet-type representations for a multifunction $X$ and for the collection of random vectors taking values in $X$, see \cite{phleps2001} for an overview of Choquet's theorem. We show that the collection of random vectors taking values in a closed convex random set is a Choquet set. Theorem~\ref{thm:ConvexClosedThenChoquet}, Corollary~\ref{cor:LpXChoquet} and Theorem~\ref{thm:ChoquetHullCharacterization} are the main results on Choquet combinations, providing the theoretical foundations for further methods in the computation of Aumann integrals.


For random vectors $\xi$ and $\zeta$, and an event $B\in\F,$ the random vector $1_B \xi+ 1_{B^c} \zeta$ is called a decomposition of $\xi$ and $\zeta.$ For every $p \in \{0\} \cup [1,+\infty),$ a subset of $L^p(\R^d)$ is called decomposable if it contains every decomposition of its elements and the decomposable hull of a subset of $L^p(\R^d)$ is defined as the smallest decomposable subset $L^p(\R^d)$ containing it. Consider a measurable multifunction $X \colon \O \to \mathcal{P}(\R^d)$. If $\xi$ and $\zeta$ are random vectors taking values in $X$ and $B\in\F$ is an event, then $1_B \xi+ 1_{B^c} \zeta$ takes values in $X,$ showing the decomposability of the collection of random vectors taking values in $X.$ Moreover, given a sequence $(\xi_i)_{i=1}^{\infty}$ of random vectors taking values in $X$ and a measurable partition $(B_i)_{i=1}^{\infty}$ of $\O$, the random vector $\sum_{i=1}^{\infty}  1_{B_i} \xi_i $, called a countable decomposition of $(\xi_i)_{i \in \N}$, also takes values in $X$. 

By generalizing finite and countable decompositions, we introduce \emph{Choquet decompositions}, which serve as uncountable-type decompositions in $L^p(\R^d)$. We call a subset of $L^p(\R^d)$ a \emph{$p$-Choquet decomposable} set if it contains every $p$-integrable Choquet decomposition of its elements, and define the \emph{$p$-Choquet decomposable hull} of a subset of $L^p(\R^d)$ as the smallest $p$-Choquet decomposable subset of $L^p(\R^d)$ containing it. We show that $p$-Choquet decomposable hull coincides with the strong closure of decomposable hull for subsets of $L^p(\R^d)$. However, by studying it as a single hull operator, we are able to extend the definition of $p$-Choquet decomposable hull for every subset of $L^0(\R^d)$, without assuming any integrability, for every $p \in \{0\} \cup [1,+\infty)$. It turns out that $p$-Choquet decomposable hull operators can be obtained from the $0$-Choquet decomposable hull operator via truncation. We also show how $p$-Choquet decomposable hull operators interact with each other when applied in succession. See Proposition~\ref{prop:pChoquetDecomposabilityChar}, Corollary~\ref{cor:pchdHull} and Theorem~\ref{thm:chpchd=pchcdh} for the main results and Theorems~\ref{thm:chpchd=pchcdhExt} and~\ref{thm:operatordiscussion} for a discussion from a hull operator point of view.

Finally, consider a measurable multifunction $X \colon \O \to \mathcal{P}(\R^d)$ with convex values. For random vectors $\xi_1,\ldots,\xi_m$ taking values in $X$ and measurable functions $\l_1,\ldots,\l_m \geq 0$ with $\sum_{i=1}^m \l_i =1$ $\p$-almost surely, the random vector $\sum_{i=1}^m \l_i \xi_i$, called a random convex combination of $\xi_1,\ldots,\xi_m$, takes values in $X$. For every $p \in \{0\} \cup [1,+\infty)$, by generalizing random convex combinations and Choquet decompositions, we introduce \emph{Choquet convex decompositions}, which serve as uncountable-type convex decompositions in $L^p(\R^d)$. We call a subset of $L^p(\R^d)$ a \emph{$p$-Choquet convex decomposable} set if it contains every $p$-integrable Choquet convex decomposition of its elements and define the \emph{$p$-Choquet convex decomposable hull} of a subset of $L^p(\R^d)$ as the smallest $p$-Choquet convex decomposable subset of $L^p(\R^d)$ containing it. 

We show that $p$-Choquet convex decomposable hull coincides with the closure of the decomposable hull of convex hull, hence also with the $p$-Choquet decomposable hull of convex hull, for subsets of $L^p(\R^d)$. Similar to what we did before, by studying it as a hull operator, we extend the definition of $p$-Choquet convex decomposable hull for every subset of $L^0(\R^d)$ for every $p \in \{0\} \cup [1,+\infty)$. It turns out that the $p$-Choquet convex decomposable hull operators can be obtained from the $0$-Choquet convex decomposable hull operator via truncation. We also show how $p$-Choquet convex decomposable hull operators interact with each other when applied in succession. See Proposition~\ref{prop:LpXisChcd} and Corollaries~\ref{cor:pchcdiffconvexdecclosed} and~\ref{cor:pchcdHull} for the main results and Corollary~\ref{cor:pchcdHullExt} and Theorem~\ref{thm:operatordiscussionConvex} for a discussion from hull operator point of view.


We review random sets in Section~\ref{sec:randomsets} focusing on graph measurability and Effros measurability. We also consider set-theoretic operations in relation to these measurability notions. Section~\ref{sec:integrals} contains selections of random sets, the Aumann integral and selection expectation. We review the representation of compact convex sets in Section~\ref{sec:compactconex}. Section~\ref{sec:choquet} introduces Choquet combinations and Section~\ref{sec:choquetdec} introduces Choquet decompositions. In Section~\ref{sec:choquetconvdec}, we introduce Chouquet convex decompositions and investigate their relation to Choquet combinations and Choquet decompositions. We give some concluding remarks in Section~\ref{sec:conc}, followed by some ideas for future research in Section~\ref{sec:fut}. Section~\ref{sec:Bochner}, the appendix, treats Bochner integrals in relation to transition kernels and distributions.

\section{Random sets}\label{sec:randomsets}

In this section, we consider multifunctions and provide a detailed review of two relevant measurability notions. 
Some are original slight improvements, we intend to have an almost complete overview for our purposes. Let $E$ be a separable metric space. Let $\mathcal{P}(E)$ denote the power set of $E$ and let $\B(E)$ denote the Borel $\sigma$-algebra of $E.$ Let $(\O,\F,\p)$ be a probability space. For $x\in E$ and $\varepsilon>0$, $B_\varepsilon(x)$ denotes the open ball centered at $x$ with radius $\varepsilon$. For a set $A \subseteq E$, we denote by $\cl A=\overline{A}$, $ \intt A$, $\partial A$ the closure, interior, boundary of $A$, respectively. We write $\overline{B}_\varepsilon(x)\coloneqq \overline{B_\varepsilon(x)}$ for the closed ball centered at $x\in E$ with radius $\varepsilon>0$. When $E$ is a Banach space, we denote by $\conv A$, $\cconv A$ the convex hull, closed convex hull of $A$, respectively.

\subsection{Measurability notions}

For a multifunction $X \colon \O \to \mathcal{P}(E),$ the \emph{domain} and \emph{graph} of $X$ are defined as
\[
\dom(X) \coloneqq \{ \o \in \O \colon X(\o) \not = \emptyset \},\qquad \gr (X) \coloneqq \{ (\o,x) \in \O \times E \colon x \in X(\o) \},
\]
respectively. For a set $B \subseteq E,$ denote $X^- (B) \coloneqq \{ \o \in \O \colon X(\o) \cap B \not = \emptyset \}$. A multifunction $X$ is said to be \emph{closed/compact/open/nonempty} if $X(\o)$ is closed/compact/open/nonempty for $\p$-almost every $\o\in\O$. When $E$ is a Banach space, $X$ is said to be \emph{convex} if $X(\o)$ is convex for $\p$-almost every $\o\in\O$. We recall the following measurability notions for multifunctions:

\begin{definition}\label{def:graphmble}
	A multifunction $X \colon \O \to \mathcal{P}(E)$ is called 
	\begin{enumerate}
		\item \emph{graph measurable} or a \emph{random set} if $\gr(X) \in \F \otimes \B (E)$, 
		\item \emph{Effros measurable} if $X^- (U) \in \F$ for every open subset $U$ of $E$. 
	\end{enumerate}
\end{definition}

We have the following relationship between graph measurability and Effros measurability:

\begin{theorem}[Theorem 4.1, Hess~\cite{hess2002}]
	\label{thm:relationBwMeasurability}
	For closed multifunctions, Effros measurability implies graph measurability. When $E$ is a separable Banach space and the probability space $(\O,\F,\p)$ is complete, graph measurability implies Effros measurability, and in this case the two are equivalent for closed multifunctions.
\end{theorem}

We consider several examples of closed Effros measurable multifunctions that are automatically graph measurable when the probability space $(\O,\F,\p)$ is complete.

\begin{example}\label{example:mainexamples}
	Let $\xi, \eta$ be real-valued random variables such that $\eta \leq \xi$. Let $r$ be a real-valued positive random variable and let $\varphi$ be an $d$-dimensional random vector with $d\in\N$. Let $X=\{ \xi \},$ $Y=(-\infty, \xi],$ $Z= [\eta,\xi] $ and  $W=\overline{B}_{r} (\varphi) \coloneqq\{x\in\R^d\colon \abs{\varphi-x}\leq r\}$, where $\abs{\cdot}$ denotes the Euclidean norm on $\R^d$. Then $X, Y, Z$ and $W$ are closed convex Effros measurable multifunctions. 
\end{example}

\subsection{Set-theoretic operations}

We start with a theorem summarizing some useful properties of Effros measurable multifunctions. 
	Note that Effros measurable multifunctions are mostly studied in the closed-valued setting and the proofs rely heavily on Castaing representations, see for instance Theorem~\ref{thm:randomclosedsetchar} below; such proofs fail to generalize further. 
	For that reason, we unify the theorems with minimal assumptions, including the proofs for those that extend the literature.
 Recall that, for a set $B\subseteq \O$, its indicator function $1_B$ is defined by $1_B(\o)=1$ for $\o\in B$ and by $1_B(\o)=0$ for $\o\in B^c\coloneqq \O\setminus B$.


\begin{theorem}[Molchanov~\cite{molchanov2005}]
	\label{thm:setTheoriticEffros}
	Let $X, Y, X_1, X_2, \ldots$ be Effros measurable multifunctions. The following results hold:
	\begin{enumerate}
		\item $\dom(X) \in \F$.
				
		\item The multifunction $\overline{X}$ is Effros measurable.
		
		\item For every $y \in E,$ the distance function $d(y,X) \coloneqq \inf \{ d(y,x) \colon x\in X \}$	is a random variable taking values in $[0,+\infty]$.
		
		\item For every $\eps >0$, the multifunction $X^{\eps} \coloneqq \{ x \in E \colon d(x,X) < \eps \}$ is Effros measurable. 

		\item For every $B \in \F,$ the multifunction $Z = X 1_B + Y 1_{B^c}$ is Effros measurable, where $Z(\o) = X(\o)$ if $\o \in B$, and $Z(\o)=Y(\o)$ if $\o \in B^c$. 
		
		\item For every measurable partition $(B_n)_{n \in \N}$ of $\O,$ the multifunction $Z = \sum_{n=1}^{\infty} X_n 1_{B_n}$ is Effros measurable. 
		
		\item The multifunction $\bigcup_{ n \in \N} X_n$ is Effros measurable. 
		
		\item If $X$ and $X^\prime$ are Effros measurable multifunctions in second countable metric spaces 
		$E$ and $E^\prime$, respectively, then the Cartesian product $X \times X^\prime$ is an Effros measurable multifunction in $E \times E^\prime$ considered with the product topology.
		\end{enumerate}
	
		Let us further assume that $E$ is a separable Banach space and $X,Y,X_1,X_2,\ldots$ are closed. Then, we also have the following results:
		
		\begin{enumerate}
		\setcounter{enumi}{8}
		\item The multifunctions $\conv X$, $\cconv X$ are Effros measurable. 
		
		\item For every real-valued random variable $\a$, the multifunction $\a X$ is closed and Effros measurable.
		
		\item The multifuncions $\cl (X^c)$, $\cl(\intt X)$, $\partial X$ are closed and Effros measurable.
		
		
		\item The multifunctions $X \cup Y$ and $X \cap Y$ are closed and Effros measurable.
		\item The multifunctions $\cl(\bigcup_{ n \in \N} X_n)$ and $\bigcap_{ n \in \N} X_n$ are closed and Effros measurable. 
		
		
		\item The multifunction $\cl (X+Y) $ is closed and Effros measurable. 
		
		\item If $X$ and $Y$ are both bounded $\p$-almost surely, then $d(X,Y) \coloneqq \inf \{ d(x,y) \colon x\in X, \, y \in Y \}$ is a real-valued random variable. 
	\end{enumerate}

\end{theorem}

\begin{proof}
		1. For $U = E,$ we have $X^-(U) = \{\o \in \O \colon X(\o) \cap E \not = \emptyset\} = \dom(X) \in \F$ by the Effros measurability of $X$.\\
		2. For every open subset $U$ of $E,$ we have
		\begin{align*}
		\overline{X} ^-(U) &= \{ \o \in \O \colon \overline{X(\o)} \cap U \not = \emptyset \} 
		= \{ \o \in \O \colon X(\o) \cap U \not = \emptyset \} 
		= X^-(U) \in \F
		\end{align*}
		by the Effros measurability of $X$. Indeed, $X$ is Effros measurable if and only if $\overline{X}$ is.\\
		3. For every $\eps >0,$ we have 
		\begin{align*}
		\{ \o \in \O \colon d(y,X) < \eps \} &=\{ \o \in \O \colon \exists x \in X(\o) \text{ such that } d(y,x) < \eps \} \\
		&=\{ \o \in \O \colon X(\o) \cap B_{\eps} (y) \not = \emptyset \} = X^-(B_{\eps} (y))  \in \F
		\end{align*}
		by the Effros measurability of $X$. 
		Indeed, the measurability of $d(y,X)$ ensures the Effros measurability of $X$ as well.\\
		4. For every open subset $U$ of $E,$ $U^{\eps} \coloneqq \{ x \in E \colon d(x,U) < \eps \}$ is an open set. Then, 
		\begin{align*}
		(X^{\eps})^-(U) &= \{ \o \in \O \colon X^{\eps} \cap U \not = \emptyset \} \\
		&= \{ \o \in \O \colon \exists x \in U \text{ such that } x \in X^{\eps} (\o) \} \\
		&= \{ \o \in \O \colon \exists x \in U, \, \exists y \in X(\o) \text{ such that } d(x,y) < \eps \} \\
		&= \{ \o \in \O \colon \exists y \in X(\o) \text{ such that } y \in U^{\eps} \} \\
		&= \{ \o \in \O \colon X \cap U^{\eps} \not = \emptyset \} = X^-(U^{\eps}) \in \F
		\end{align*}
		by the Effros measurability of $X$.\\
		5. For every open subset $U$ of $E,$ we have
		\begin{align*}
		Z^-(U) &= \{ \o \in \O \colon Z(\o) \cap U \not = \emptyset \} \\
		&= \Big( \{ \o \in \O \colon Z(\o) \cap U \not = \emptyset \} \cap B \Big) \cup \Big(  \{ \o \in \O \colon Z(\o) \cap U \not = \emptyset \} \cap B^c \Big) \\
		&= \Big( \{ \o \in \O \colon X(\o) \cap U \not = \emptyset \} \cap B \Big) \cup \Big( \{ \o \in \O \colon Y(\o) \cap U \not = \emptyset\} \cap B^c \Big) \\
		&= \big( X^-(U) \cap B \big) \cup \big( Y^-(U) \cap B^c \big) \in \F
		\end{align*}
		by the Effros measurability of $X$ and $Y$.\\		
		6. For every open subset $U$ of $E,$ we have
		\begin{align*}
		Z^-(U)&= \{ \o \in \O \colon Z(\o) \cap U \not = \emptyset \} 
		= \bigcup_{n =1}^{\infty} \{ \o \in \O \colon Z(\o) \cap U \not = \emptyset ,  \,\o \in B_n \} \\
		&= \bigcup_{n =1}^{\infty} \{ \o \in \O \colon X_n(\o) \cap U \not = \emptyset ,  \,\o \in B_n \} 
		= \bigcup_{n =1}^{\infty} (X_n ^-(U) \cap B_n ) \in \F
		\end{align*}
		by the Effros measurability of $X_n$, $n\in\N$.	\\
		7. For every open subset $U$ of $E,$ we have 
		\begin{align*}
		\of{\bigcup_{ n \in \N} X_n}^-(U) &=  \cb{ \o \in \O \colon \bigcup_{ n \in \N} X_n (\o) \cap U \not = \emptyset } \\
		&= \bigcup_{ n \in \N} \{ \o \in \O \colon X_n (\o) \cap U \not = \emptyset \} = \bigcup_{ n \in \N} X_n ^-(U) \in \F
		\end{align*}
		by the Effros measurability of $X_n$, $n\in\N$. \\ 
		8. Let $(U_i)_{i \in \N}$ be a basis for $E$ and let $(V_j)_{j \in \N}$ be a basis for $E^\prime$. Then, $(U_i \times V_j)_{i,j \in \N}$ is a basis for $E \times E^\prime$. We have
		\begin{align*}
		(X \times X^\prime) ^-(U_i \times V_j) &=  \{ \o \in \O \colon (X \times X^\prime) \cap (U_i \times V_j)\not = \emptyset\} \\
		&= \{ \o \in \O \colon X \cap U_i \not = \emptyset\} \cap \{ \o \in \O \colon X^\prime \cap V_j \not = \emptyset\}\\
		&= X^- (U_i) \cap (X^\prime)^- (V_j) \in \F
		\end{align*}
		by the Effros measurability of $X$ and $X^\prime$. By second countability, we have $(X\times X^\prime) ^-(W) \in \F$ for every open subset $W$ of $E \times E^\prime$.\\
	We refer the reader to Theorem 1.3.25 in Molchanov~\cite{molchanov2005} for the proofs of 9-15.
\end{proof}

We continue with a lemma that is known as the \emph{projection theorem}.
 
\begin{lemma}[Theorem F.7, Molchanov~\cite{molchanov2005}]\label{lem:projection}
	Suppose that $E$ is a separable Banach space and $(\O,\F,\p)$ is a complete probability space. Then, the projection of every set in $\F \otimes \B(E)$ onto $\O$ is $\F$-measurable. 
	
	
\end{lemma}

The next theorem summarizes some useful properties of graph measurable multifunctions, where we include the proofs for those extending the literature.

\begin{theorem}[Molchanov~\cite{molchanov2005}]
	\label{thm:setTheoriticgraph}
	Let $X, Y, X_1, X_2,\ldots$ be graph measurable multifunctions. Then, the following results hold.  
	\begin{enumerate}
		
		\item The multifunction $X^c$ is graph measurable.
		
		\item The multifunctions $X \cup Y,$ $X \cap Y,$ $Y \setminus X$ and $X \triangle Y \coloneqq (X\setminus Y) \cup (Y\setminus X)$ are graph measurable.

		\item The multifunctions $\bigcup_{ n \in \N} X_n$ and $\bigcap_{ n \in \N} X_n$ are graph measurable. 
		
		\item For every $B \in \F,$ the multifunction $Z = X 1_B + Y 1_{B^c}$ is graph measurable. 
		
		\item For every measurable partition $(B_n)_{n \in \N}$ of $\O,$ the multifunction $Z = \sum_{n=1}^{\infty} X_n 1_{B_n}$ is graph measurable. 
		
		\item If $X$ and $X^\prime$ are graph measurable multifunctions in second countable spaces $E$ and $E^\prime$, then the $X \times X^\prime$ is a graph measurable multifunction in $E \times E^\prime$ considered with the product topology.
		\end{enumerate}

		Let us further assume that $E$ is a separable Banach space and $X, Y, X_1, X_2, \ldots$ are closed for every $n \in \N$. Then, the following results are valid:
		
		\begin{enumerate}
			\setcounter{enumi}{6}
		\item $\dom(X) \in \F$.
		
		\item The multifunction $\overline{X}$ is graph measurable, equivalently, it is Effros measurable. 
		
		\item For every $y \in E,$ the distance function $d(y,X) \coloneqq \inf \{ d(y,x) \colon x\in X \}$	is a random variable taking values in $[0,+\infty].$
		
		\item For every $\eps >0,$ the multifunction $X^{\eps} \coloneqq \{ x \in E \colon d(x,X) < \eps \}$ is graph measurable. 
		
		
	\end{enumerate}
\end{theorem}

\begin{proof} 
		1. We have $\gr(X^c) = ( \gr X )^c \in \F \otimes \B(E)$	by the graph measurability of $X$.\\
		2. We have $ \gr(X \cup Y) = \gr(X) \cup \gr(Y) \in \F \otimes \B(E)$	and $\gr(X \cap Y) = \gr(X) \cap \gr(Y) \in \F \otimes \B(E)$ by the graph measurability of $X$ and $Y$. The graph measurability of $Y \setminus X$ and $X \triangle Y$ follows easily. \\
		3. We have
		\[
		\gr\of{\bigcup_{ n \in \N} X_n} = \bigcup_{ n \in \N} \gr(X_n) \in \F \otimes \B(E)
		\]
		and 
		\[
		\gr\of{\bigcap_{ n \in \N} X_n} = \bigcap_{ n \in \N} \gr(X_n) \in \F \otimes \B(E) 
		\]
		by the graph measurability of $X_n$, $n\in\N$.\\
		4. We have 
		\[
		\gr (Z) = \Big( \gr (X) \cap (B \times E) \Big) \cup \Big(\gr(Y) \cap (B^c \times E) \Big) \in \F \otimes \B(E) 
		\]
		by the graph measurability of $X$ and $Y$.\\
		5. We have 
		\[
		\gr (Z) = \bigcup_{n =1}^{\infty} \Big( \gr(X_n) \cap (B_n \times E) \Big) \in \F \otimes \B(E) 
		\]
		by the graph measurability of $X_n$.\\		
		6. We have 
		\[
		\gr (X \times X^\prime) = \Big( \gr(X) \times E^\prime \Big) \cap \Big( \gr(X^\prime)\times E\Big)  \in \F \otimes \B(E) \otimes \B(E^\prime) 
		= \F \otimes \B(E \times E^\prime) 
		\]
		by the graph measurability of $X$ and $X^\prime$. \\	
		7. The projection of $\gr (X) $ onto $\O$ is $\{ \o \in \O \colon X(\o) \not = \emptyset \}$, hence $\dom(X) \in \F \otimes \B(E)$ by Lemma~\ref{lem:projection}.\\
%
%
We refer the reader to Theorem 1.3.25 in Molchanov~\cite{molchanov2005} for the proofs of 8-10.
\end{proof}

Next, we consider sampling points from random sets.

\section{Selections of random sets and Aumann integral}\label{sec:integrals}

In this section, we review the notion of expectation for random sets, which is defined by considering the expectations of all samples from the random set. Let $E$ be a separable Banach space equipped with its Borel $\sigma$-algebra $\B(E)$ and let $(\O,\F,\p)$ be a probability space. We denote by $L^0(E)$ the collection of all equivalence classes of random variables taking values in $E$, where two random variables belong to the same class if they are equal $\p$-almost surely. For each $p\in [1,\infty)$, $L^p(E)$ denotes the set of all $\xi\in L^0(E)$ with $\norm{\xi}_p\coloneqq(\E\abs{\xi}^p)^{1/p}<+\infty$, where $\abs{\cdot}$ is the norm on $E$. When $\xi\in L^p(E)$, $\E(\xi)\in E$ denotes the expectation of $\xi$, which is defined as a Bochner integral. We refer the reader to Frieler and Knoche~\cite{frieler2001}, Mikusi\'{n}ski~\cite{Mikusinski1978} and Rieffel~\cite{Rieffel1968} for an overview of Bochner integration over Banach spaces.

\subsection{Measurable and integrable selections}

\begin{definition}\label{def:selection}
	Let $X \colon \O \to \mathcal{P}(E)$ be a multifunction. A random variable $\xi$ with values in $E$ is called a \emph{measurable selection} of $X$ if $\xi(\o) \in X(\o)$ for every $\o \in \dom(X)$.  
\end{definition}

Note that we do not have any measurability assumption on $X$, hence $\dom(X)$ need not be measurable. When $\dom(X)$ is ensured to be measurable, it suffices to have $\xi(\o) \in X(\o)$ for $\p$-almost every $\o\in \dom(X)$.

We denote the collection of all equivalence classes of measurable selections of $X$ by $L^0(X)$. Let $p \in [1,\infty)$. A measurable selection $\xi$ of $X$ is called \emph{$p$-integrable} if $\xi \in L^p(E)$, in particular, \emph{integrable} if $\xi \in L^1(E)$. We denote the collection of all $p$-integrable selections of $X$ by $L^p(X)$. Note that $L^p(X) = L^0(X) \cap L^p(E).$ 

The following theorem provides with an affirmative answer to existence of measurable selections:


\begin{theorem}[Theorem 4.4, Hess~\cite{hess2002}]\label{thm:existenceofselections}
	Consider a multifunction $X \colon \O \to \mathcal{P}(E)$. 
	\begin{enumerate}
		\item If $X$ is closed and Effros measurable with $\dom(X) \not = \emptyset$, then $X$ admits a measurable selection.
		\item If $(\O,\F,\p)$ is a complete probability space and $X$ is graph measurable with $\gr(X) \not = \emptyset$, then $X$ admits a measurable selection.
	\end{enumerate}
\end{theorem}

Closed Effros measurable multifunctions can be characterized by a countable family of selections as the next theorem shows.

\begin{theorem}[Theorem~4.5, Hess~\cite{hess2002}]\label{thm:randomclosedsetchar}
	Consider a closed multifunction $X \colon \O \to \mathcal{P}(E).$ Then $X$ is Effros measurable if and only if there exists a sequence 
		$(\xi_n)_{n \in \N}$ 
	of random elements in $E$ such that 
	\begin{equation}\label{Castaing}
	X(\o) = \cl\{\xi_n(\o) \colon n\in \N \} 
	\end{equation}
	for every $\o \in \dom(X)$ and $\dom(X) \in \F$.
\end{theorem}

\begin{remark}\label{rem:hess}
	In \cite{hess2002}, the previous theorem is stated without the condition $\dom(X)\in\F$. However, without this condition, the theorem seems to be invalid: Let $\xi$ be a random variable with values in $E$ and let $B \subseteq \O$ be a nonmeasurable set. Then, consider $X(\o) \coloneqq \{ \xi(\o) \}$ for $\o\in B$ and $X(\o) \coloneqq \emptyset$ for $\o\in B^c$. Then, $\dom(X) = B$ and $X(\o) = \cl \{ \xi(\o) \} $ for each $\o\in \dom(X)$. However, $X$ is not Effros measurable: $X^-(E) = \dom(X) \not \in \F $. We provide a corrected version in Theorem~\ref{thm:randomclosedsetchar}. Nevertheless, we skip the proof since the rest of the arguments in \cite{hess2002} still works.
	\end{remark}


The representation of a multifunction $X$ as in \eqref{Castaing} is called a \emph{Castaing representation} of $X$. By providing a practical approach for closed Effros measurable multifunctions, Theorem~\ref{thm:randomclosedsetchar} plays a key role in Theorem~\ref{thm:selections} below.

For every $p \in \{0\} \cup [1,+\infty)$, a set $A \subseteq L^p(E)$ of random variables is called \emph{decomposable} if $1_B \xi+ 1_{B^c} \zeta \in A$ for every $\xi,$ $\zeta \in A$ and $B \in \F.$ It is easy to see that $L^p(E)$ is decomposable and an arbitrary intersection of decomposable sets is decomposable. Hence, for a set $A \subseteq L^p(E)$, we may define the \emph{decomposable hull $\dec A$} of $A$ as the intersection of all decomposable subsets of $L^p(E)$ containing $A,$ which turns out to be the smallest decomposable set containing $A$. 

\begin{remark}
	There is a simple quantitative characterization of the decomposable hull. For a set $A \subseteq L^p(E),$ it is easy to see that
	\[
	\dec (A) =\Big \{ \sum_{i=1}^m  1_{B_i} \xi_i \colon \xi_i \in A, \, (B_i)_{i=1}^m \, \text{ is a measurable partition of } \O, m \in \mathbb{N} \Big  \} .
	\]
	Here, we call $\sum_{i=1}^m  1_{B_i} \xi_i $ the \emph{decomposition} of $\xi_1,\ldots,\xi_m$ along the partition $(B_i)_{i=1}^m$.
\end{remark}

\begin{example}
	\label{ex:unitball}
	If $(\O,\F,\p)$ is nontrivial, 
then it is clear that the unit ball $B_1(0)$ in $L^1(\R^d)$ is not decomposable: Let $B \in \F$ be an event with $0 < \p(B) < 1.$ Then $\xi \coloneqq 1_B \frac{3}{4 \p(B)} \in B_1(0)$ and $\zeta \coloneqq 1_{B^c} \frac{3}{4 \p(B^c)} \in B_1(0)$, yet 
\[ 
\xi 1_B + \zeta 1_{B^c} = 1_B \frac{3}{4 \p(B)} + 1_{B^c} \frac{3}{4 \p(B^c)} \not \in B_1(0).
\] 
Indeed, one has $\dec B_1(0) = L^1(\R^d)$. The same holds for the closed unit ball in $L^1(\R^d)$. 
\end{example}

It is clear that the weak and strong closures of a decomposable subset of $L^p(E)$ are decomposable. The next proposition shows that convex hull and decomposable hull operators are compatible with each other. 

\begin{proposition}\label{prop:convdec=decconv}
	The decomposable hull of a convex subset of $L^p(E)$ is convex, and the convex hull of a decomposable subset of $L^p(E)$ is decomposable. Moreover, for every $A \subseteq L^p(E),$ one has
	\begin{equation}\label{convdec}
	\conv \dec A = \dec \conv A. 
	\end{equation}
\end{proposition}

\begin{proof}
	For the first part, let $A$ be a convex subset of $L^p(E).$ Let $u,v \in \dec A$ and let $0 < \l < 1.$ We  show that $\l u + (1-\l)v \in \dec A$ to conclude the convexity of $\dec A \colon$ since $u, v \in \dec A$, they can be written as the decompositions of elements of $A$ along measurable finite partitions of $\O$. Let $(B_i)_{i=1}^m$ be a finer partition so that $u = \sum_{i=1}^m \xi_i  1_{B_i}$ and $v = \sum_{i=1}^m \zeta_i  1_{B_i}$ for some $(\xi_i)_{i=1}^m ,$ $(\zeta_i)_{i=1}^m \subseteq A.$ Then,
	\[
	\l u + (1-\l)v = \l \sum_{i=1}^m \xi_i  1_{B_i} + (1-\l) \sum_{i=1}^m \zeta_i  1_{B_i} = \sum_{i=1}^m \big( \l  \xi_i + (1-\l) \zeta_i \big) 1_{B_i} \in \dec A
	\]
	where $\l  \xi_i + (1-\l) \zeta_i  \in A$ for every $i \in \{1,\ldots, m\}$ by the convexity of $A$. Hence, $\dec A$ is convex. 
	
	For the second part, let $A$ be a decomposable subset of $L^p(E).$ Let $u,v \in \conv A$ and let $B \in \F$. We show that $ u 1_B + v 1_{B^c}\in \conv A$ to conclude the decomposabilty of $\conv A$. Since $u,v \in \conv A$, we have $ u = \sum_{i=1}^n \a_i \xi_i $, $ v = \sum_{j=1}^m \b_j \zeta_j $ for some $(\xi_i)_{i=1}^n ,$ $(\zeta_i)_{i=1}^m \subseteq A$. One can find constants $(c_{ij})_{i,j}$ with $\sum_{j=1}^m c_{ij} = \a_i$ for every $i \in \{1,\ldots,n\}$ and $\sum_{i=1}^n c_{ij} = \b_j$ for every $j \in \{1,\cdots,m\}$ so that
	\[
	u 1_B + v 1_{B^c} = \sum_{i=1}^n \alpha_i \xi_i 1_B + \sum_{j=1}^m \beta_j \zeta_j 1_{B^c} 
= \sum_{i,j}  c_{ij} (\xi_i 1_B + \zeta_j 1_{B^c}) \in \conv A,
	\]
	where $\xi_i 1_B + \zeta_j 1_{B^c} \in A$ by the decomposability of $A$. Hence, $\conv A$ is decomposable. 
	
	As a consequence of the first two parts, \eqref{convdec} follows for every $A \subseteq L^p(E)$.
\end{proof}

The next theorem establishes several useful properties of $p$-integrable selections of a random set. We denote by $\cl_{p}$ the closure operator in the norm topology for $p\in [1,\infty)$, and in the topology of convergence in probability for $p=0$. 

\begin{theorem}[Molchanov~\cite{molchanov2005}]
	\label{thm:selections} 
	Let $p \in \{0\} \cup [1,+\infty)$ and consider two multifunctions $X,Y \colon \O \to \mathcal{P}(E)$. The following results hold:
	\begin{enumerate}
		\item \label{thm:selection1} $L^p(X)$ is decomposable.
		\item \label{thm:selection2} If $X$ is closed, then $L^p(X)$ is a closed subset of $L^p(E)$ in the strong topology.
		\item \label{thm:selection3} Suppose that $X$ is closed with $L^p(X) \neq \emptyset$. Then, $X$ is Effros measurable if and only if there exists a sequence $( \xi_n )_{n\in\N}$ of random elements in $E$ such that 
		\[
		X(\o) = \cl \{ \xi_n(\o) \colon  n \in \N \} 
		\]
		for every $\o \in \dom(X)$ and $\dom(X) \in \F$, that is, $X$ has a $p$-integrable Castaing representation. 
		
		\item \label{thm:selection4} Suppose that $X$ is closed and Effros measurable with $L^p(X) \neq \emptyset$. Then, for every $p$-integrable Castaing representation $(\xi_n)_{n\in \N}$ of $X$ 
		we have
		\[
		L^p(X) = \cl_{p} \dec \{ \xi_n  \colon  n \in \N \} .
		\]
		
		\item \label{thm:selection5} Let $X$ be closed and Effros measurable. If $L^p(X) \not = \emptyset ,$ then
		\[
		L^p(\cconv X) = \cconv L^p(X).
		\]
		Moreover, $X$ is convex if and only if $L^p(X)$ is convex.
		
		\item \label{thm:selection6} If $X$ and $Y$ are closed and Effros measurable with $L^p(X)=L^p(Y) \neq \emptyset,$ then $X=Y$ almost surely. 
		
		\item \label{thm:selection7} Suppose that $(\O,\F,\p)$ is a complete probability space. If $X$ and $Y$ are graph measurable with $L^p(X) = L^p(Y) \not = \emptyset,$ then $X = Y$ almost surely.
		
		\item \label{thm:selection8} Let $\emptyset \not =A \subseteq L^p(E)$ be a closed set. Then, $A$ is decomposable if and only if $A = L^p(Z)$ for some closed and Effros measurable multifunction $Z \colon \O \to \mathcal{P}(E)$. In this case, $A$ is convex if and only if $Z$ is convex. 
		
		\item \label{thm:selection9} If $X$ is closed, then $L^p(X)= L^p(Z)$ for some closed and Effros measurable multifunction $Z \colon \O \to \mathcal{P}(E)$. In this case, $X$ is convex if and only if $Z$ is convex. 
		
		\item \label{thm:selection10} Suppose that $X$ is closed and Effros measurable such that $L^1(X)$ is bounded in $L^1(E)$. Then, $X$ is 
		relatively weakly compact if and only if $L^1(X)$ is relatively weakly compact in $L^1(E)$.  
		
		\item \label{thm:selection11} Suppose that $X$ is convex, closed and Effros measurable such that $L^1(X)$ is bounded in $L^1(E)$. Then, $X$ is 
		weakly compact if and only if $L^1(X)$ is weakly compact in $L^1(E)$. 
		
		\item \label{thm:selection12} Suppose $X$ is closed and Effros measurable such that $L^p(X)$ is bounded in $L^p(E)$. Then, $L^0(X) = L^s(X) = L^p(X)$ for every $s \in [1,p]$.
		
	\end{enumerate}
\end{theorem}

\begin{proof}
	We refer the reader to Theorem 2.1.10 for~\ref{thm:selection1}, to Proposition 2.1.4 for~\ref{thm:selection2},~\ref{thm:selection3},~\ref{thm:selection6},~\ref{thm:selection12} to Lemma 2.1.5 for~\ref{thm:selection4}, to Proposition 2.1.7 and Corollary 2.1.11 for~\ref{thm:selection5},~\ref{thm:selection8},~\ref{thm:selection9}, to Theorem 2.1.17 for~\ref{thm:selection10}, and to Theorem 2.1.18 for~\ref{thm:selection11}, all cited results being in \cite{molchanov2005}. \ref{thm:selection7} follows from Theorems~\ref{thm:setTheoriticgraph} and~\ref{thm:existenceofselections}.
\end{proof}

There is an appealing connection between decomposability, weak closedness and convexity given by the next theorem.

\begin{theorem}[Theorem II.3.17, Hu~\cite{hu2013}]
	\label{thm:weaklyclosedDecIsConv} 
	If $(\O,\F,\p)$ is non-atomic, then every decomposable weakly closed subset of $L^p(E)$ is convex for every $p \in [1,+\infty).$
\end{theorem}

Based on Example~\ref{ex:unitball}, we observe that $L^1(X) = L^1(\R^d)$ whenever $L^1(X)$ has nonempty interior, which happens only in the case where $X = \R^d ,$ $\p$-almost surely. Hence, in most cases of interest $L^1(X)$ has empty interior. 

Theorem~\ref{thm:selections} provides a practical way of studying the collection of $p$-integrable selections of a random set. We next focus on expectations of integrable selections. 





\subsection{Aumann integral and selection expectation}

\begin{definition}
	\label{def:Integral}
	For a multifunction $X \colon \O \to \mathcal{P}(E)$, the \emph{Aumann integral} $\int X\, d\p $ of $X$ is defined as
	\[
	\int X\, d\p \coloneqq \{\E \xi \colon \xi \in L^1(X)\},
	\]
	the \emph{selection expectation} $\E X$ of $X$ is defined as
	\[
	\E X \coloneqq \cl \int X\, d\p = \cl \{\E \xi \colon \xi \in L^1(X)\},
	\]
	where the closure is taken with respect to the norm topology on $E.$ 
\end{definition}

In general the Aumann integral is not a closed set. 

\begin{example}
	\label{ex:hyperbola}
	Let $E = \R^2$ and consider the deterministic closed multifunction 
	\[
	Z= \cb{ (x,y)\in \R^2 \colon y \geq \frac{1}{x^2} }.
	\]
	Then, it is clear that 
	\[
	L^1(Z) = \cb{ (X,Y)\in L^1(\R^2) \colon Y \geq \frac{1}{X^2} },
	\]
	and
	\[
	\int Z \, d\p = \{ (x,y)\in\R^2 \colon y > 0 \},\qquad 
	\E Z = \cl \int Z \, d\p = \{ (x,y)\in\R^2 \colon y \geq 0 \}.
	\]
	In particular, $\int Z \, d\p$ is not closed and $\int Z \, d\p \subsetneq \E Z$.
\end{example}

Note that the lack of convexity in Example~\ref{ex:hyperbola} is not the main reason for this issue:
\begin{example}
	\label{ex:triangle}
	Let $E = \R^2$ and consider the closed convex mutlifunction 
	\[
	Z = \conv \cb{ (0,0), \of{1,\frac{1}{u}}, \of{1,-\frac{1}{u}} },
	\]
	where $u$ is a uniformly distributed random variable over the interval $(0,1].$ 
	Then, we have
	\[
	L^1(Z) = \cb{ \of{\a + \b ,\frac{\a-\b}{u}}  \in L^1(\R^2) \colon \a, \b \geq 0, \, \a + \b \leq 1   },
	\]
	and
	\[
	\int Z \, d\p = \{(0,0)\} \cup \{ (x,y) \colon 0 < x < 1 \},\qquad 
	\E Z = \cl \int Z \, d\p =\{ (x,y) \colon 0 \leq  x \leq 1 \}.
	\]
	In particular, $\int Z \, d\p$ is not closed and $\int Z \, d\p \subsetneq \E Z.$
\end{example}

Yet there are cases where the Aumann integral and the selection expectation coincide, given by the next theorem. To that end, a Banach space $E$ is said to have the Radon-Nikodym property if for every finite measure space $(\A,\mathfrak{a},\mu)$ and every $E$-valued measure $v$ on $(\A,\mathfrak{a})$ which is of bounded variation and absolutely continuous with respect to $\mu,$ there exists a Bochner integrable density $f \colon \O \to E$ such that 
\[
v(B) = \int_B f \, d\mu 
\]
for every $B \in \mathfrak{a}$. It is well-known that reflexive spaces have the Radon-Nikodym property. 

\begin{theorem}[Theorem 2.1.37, Molchanov~\cite{molchanov2005}]
	\label{thm:aumannAndSeletion}
	Let $X \colon \O \to \mathcal{P}(E)$ be a closed Effros measurable multifunction such that $L^1(X)$ is bounded in $L^1(E)$. Then $\int X \, d\p$ is closed, hence it coincides with $\E X$, if one of the following conditions is satisfied: 
	\begin{enumerate}
		\item $E$ is finite dimensional.
		\item $E$ has the Radon-Nikodym property, and $X$ is convex and compact. 
		\item $X$ is convex and weakly compact. In this case, $\E X = \int X \, d\p$ is weakly compact as well.
		\item $E$ is reflexive and $X$ is convex. 
	\end{enumerate}
\end{theorem}

We continue with a theorem that investigates the convexity of Aumann integral and selection expectation.

\begin{theorem}[Proposition 2.1.15, Theorem 2.1.30, Theorem 2.1.31, Molchanov~\cite{molchanov2005}]
	\label{thm:aumannAndSeletionConvex}
	Let $X \colon \O \to \mathcal{P}(E)$ be a multifunction such that $L^1(X) \not = \emptyset$. 
	\begin{enumerate}
		\item If $X$ is graph measurable, then $\E X = \E \cl X$ and $\int \cl X \, d\p \subseteq \cl (\int X \, d\p)$. 
		\item If $X$ is convex, then $\int X \, d\p$ is convex. 
		\item If $(\O,\F,\p)$ is non-atomic and $X$ is closed, then $\E X$ is convex. 
		\item If $X$ is closed, then $\E \cconv X = \cconv \E X$.
		\item If $(\O,\F,\p)$ is non-atomic and $X$ is closed, then $\E \cconv X = \E X$.
	\end{enumerate}
\end{theorem}

In the spirit of Theorems~\ref{thm:aumannAndSeletion} and~\ref{thm:aumannAndSeletionConvex}, we consider the case $E = \R^d$ next.

\begin{theorem}[Theorem 2.1.26, Molchanov~\cite{molchanov2005}]
	\label{thm:AumannIntSelectionRd}
	Let $X \colon \O \to \mathcal{P}(\R^d)$ be a closed Effros measurable multifunction such that $L^1(X) \not = \emptyset.$ Then, $L^1(X)$ is bounded in $L^1(\R^d)$ if and only if $\E X (=\int X \, d\p)$ is compact in $\R^d.$ If the probability space $(\O,\F,\p)$ is non-atomic and $L^1(X)$ is bounded in $L^1(\R^d),$ then
	\[
	\int X \, d\p = \E X = \int \conv X \, d\p = \E \conv X.
	\]
\end{theorem}

We consider some examples:  

\begin{example}\label{example:mainexamplesIntegral}
	As in Example~\ref{example:mainexamples}, let $\xi, \eta$ be integrable real valued random variables such that $\eta \leq \xi$. Let $r$ be an integrable real-valued positive random variable and let $\varphi$ be a $d$-dimensional integrable random vector. Let $X=\{ \xi \}$, $Y=(-\infty, \xi]$, $Z= [\eta,\xi]$ and $W=\overline{B}_{r} (\varphi)$. Then, it is easy to see that
	\begin{align*}
	& L^1(X) = \{ \xi \},\quad 
	\E X = \int X \, d\p = \{ \E \xi \},\\
	& L^1(Y) = \{ \zeta \in L^1(\R) \colon \zeta \leq \xi  \},\quad 
	\E Y = \int Y \, d\p = (-\infty, \E \xi],\\
	& L^1(Z) = \{ \zeta \in L^1(\R) \colon \eta \leq \zeta \leq \xi \}, \quad \E Z = \int Z \, d\p = [\E \eta, \E \xi],\\
	& 	L^1(W) = \{ \zeta \in L^1(\R^d) \colon \abs{\zeta - \varphi} \leq r \}.
	\end{align*}
	For every $\zeta \in L^1(W),$ since $\abs{\zeta - \varphi } \leq r$, one has $\abs{\E\zeta - \E\varphi}\leq \E\abs{\zeta-\varphi}\leq \E r$, that is, $\E \zeta \in \overline{B}_{\E r} (\E \varphi)$. Hence, $\int W \, d\p \subseteq \overline{B}_{\E r} (\E \varphi).$
	Conversely, for every deterministic unit vector $u \in \R^d$, we have $\zeta \coloneqq \varphi + r  u \in L^1(W)$ and $\E \zeta = \E \varphi + \E r \, u \in \int W \, d\p$. Hence, $\partial \overline{B}_{\E r} (\E \varphi) \subseteq \int W \, d\p $ and by the convexity of $\int W \, d\p $, we conclude that 
	\[
	\E W = \int W \, d\p = \overline{B}_{\E r} (\E \varphi). 
	\]
\end{example}

We have the following theorem covering the deterministic case:


\begin{theorem}
	\label{thm:deterministicCase}
	Let $X \colon \O \to \mathcal{P}(E)$ be a deterministic multifunction. 
	\begin{enumerate}
		\item If the probability space $(\O,\F,\p)$ is non-atomic, then 
		\[
		\E X = \E \cl X = \E \conv X = \E \cconv X = \int \cconv X \, d\p = \cconv X.
		\]
		\item In particular, if the probability space $(\O,\F,\p)$ is non-atomic and $X$ is convex, then 
		\[
		\E X = \E \cl X = \int \cl X \, d\p = \cl X.
		\]
		\item If $X$ is convex and closed, then $\E X = \int X \, d\p = X.$
	\end{enumerate}
\end{theorem}

\begin{proof}
		1. Suppose that $(\O,\F,\p)$ is non-atomic. Since $X$ and $\conv X$ are graph measurable, by Theorem~\ref{thm:aumannAndSeletionConvex}, we have $\E X = \E \cl X$ and $\E \conv X = \E \cconv X$. In addition, $\cl X$ is closed and since $(\O,\F,\p)$ is non-atomic, by Theorem~\ref{thm:aumannAndSeletionConvex}, we have $\E \cconv X = \E \cconv \cl X = \E \cl X$. By taking deterministic selections, it is clear that $\int \cconv X \, d\p \supseteq \cconv X.$ Lastly, by Theorem~\ref{thm:ConvexClosedThenChoquet}, 
		 we have $ \E \cconv X \subseteq \cconv X$, completing the proof. \\
		2. This part follows easily from the first part. \\
		3. Suppose that $X$ is convex and closed. By taking deterministic selections, it is clear that $\int X \, d\p \supseteq X$. By Theorem~\ref{thm:ConvexClosedThenChoquet}, 
 we have $ \E X \subseteq  X$, completing the proof. 
\end{proof}

\begin{remark}
	Even for a deterministic multifunction $X \colon \O \to \mathcal{P}(E),$ although the selection expectation $\E X = \cconv X$ is well-understood, the Aumann integral $\int X \, d\p$ is not known in general. 
\end{remark}

Theorems~\ref{thm:aumannAndSeletion} and~\ref{thm:aumannAndSeletionConvex} suggest that it is useful to have a convex representation for the collection of integrable selections of a random set. 
Hence, we continue by investigating the representation of compact convex sets.

\section{Representation of compact convex sets}\label{sec:compactconex}

Let $V$ be a separable Banach space equipped with the Borel $\sigma$-algebra $\B(V)$ and let $V^*$ denote the dual of $V$. A point $x$ in a convex set $A\subseteq V$ is called an \emph{extreme point} of $A$ if $x \not \in \conv (A \setminus \{x\})$, that is, if $x = \l y + (1-\l) z$ with $y,$ $z \in A$ and $\l \in (0,1)$, then $y=z=x$. Let $\eps(A)$ denote the set of extreme points of $A$.

Representation of compact convex sets is well-understood in the finite-dimensional case$\colon$ every point is a finite convex combination of the extreme points. Considering all convex combinations is not sufficient to capture all points in the infinite-dimensional case, but considering all limits of all convex combinations suffices as stated in Krein-Milman Theorem:

\begin{theorem}[Theorem 8.14, Simon~\cite{simon2011}]\label{thm:kreinmilman} 
	Let $A$ be a compact convex subset of $V$. Then, $\eps(A) \not =\emptyset$ is a Baire $G_{\delta}$ set in the sense that $\eps(A) $ is a countable intersection of open sets and
	\[
	A = \cconv \eps(A). 
	\]
\end{theorem}

An appealing idea is to consider the limit of convex combinations as an integral. A probability measure $\mu$ on $(V,\B(V))$ is said to have a \emph{barycenter} if there is a point $y \in V$ satisfying 
\[
L(y) = \int_V L(x) \, \mu(dx) 
\]
for every $L \in V^*.$ Since $V^*$ separates points in $V$, a probability measure $\mu$ can have at most one barycenter, which is denoted by $r(\mu)$ or $\int_V x \, \mu(dx)$ whenever it exists. For a Borel subset $A$ of $V,$ let $\M (A)$ denote the set of all regular Borel probability measures $\mu$ with $\mu(A)=1$ that have barycenters. For every subset $A$ of $V$ that is not Borel, we define
\[
\M (A) \coloneqq \bigcup \limits_{\substack{ B\colon B \subseteq A \\ B \in\B(V) }} \M(B)
\]
which coincides with the original definition for Borel subsets. As shown by the next theorem, considering barycenters suffices to recover the limits of all convex combinations. For a point $x\in E$, $\delta_x$ denotes the Dirac measure associated to $x$, that is, $\delta_x(B)=1_B(x)$ for every $B\in \B(V)$.fg

\begin{theorem}[Theorem 9.1, Simon~\cite{simon2011}]\label{thm:barycenter}
	Let $A$ be a compact convex subset of $V$. Let $\mu$ be a regular Borel probability measure on $A$. Then, $\mu \in  \M (A),$ that is, $\mu$ has a barycenter, and $r(\mu) \in A.$ Moreover, the map $r \colon \M(A) \to A$ is a continuous affine map from $\M (A)$ (with the weak-$*$ topology) onto $A$ and is the unique such map with $r(\delta_x) = x$ for each $x \in A.$ More generally, for every closed subset $B$ of $A,$ we have
	\[ 
	r \of{\M(B)} = \cconv B.
	\]
\end{theorem}

As an immediate corollary, we obtain an integral representation of a compact convex set over the closure of its extreme points, which is known as Strong Krein-Milman Theorem:

\begin{theorem}[Theorem 9.2, Simon~\cite{simon2011}]
	\label{thm:strongKreinMilman}
	Let $A$ be a compact convex subset of $V$. Then, every point in $A$ can be represented as the barycenter of a measure in $\M ( \overline{\varepsilon(A)})$, that is, 
	\[ 
	r \Big( \M \big( \, \overline{\varepsilon(A)} \, \big) \Big) = A.
	\]
\end{theorem}

\begin{proof}
	By Krein-Milman Theorem (Theorem~\ref{thm:kreinmilman}), we have $\cconv( \overline{\varepsilon(A)} )=A$. By taking $B =\overline{\varepsilon(A)} $ in Theorem~\ref{thm:barycenter}, we obtain
	\[ 
	r \Big( \M \big( \, \overline{\varepsilon(A)} \, \big) \Big) = \cconv\big(\, \overline{\varepsilon(A)}\, \big)=A,
	\]
	as desired.
\end{proof}

The potential weakness of Theorem~\ref{thm:strongKreinMilman} is that one might have $\overline{\varepsilon(A)} = A,$ in which case the conclusion becomes $\cconv A = A$ for a compact convex subset $A$ of $V,$ which is trivial. 

\begin{example}
	Let $p \in (1,\infty)$ and let $A$ be the unit ball in $L^p(\R^d).$ By Banach-Alaoglu theorem, $A$ is weakly compact and convex. We claim that
	\[ 
	\varepsilon(A) = \{ z \in L^p(\R^d) \colon \norm{z}_p =1 \}.
	\]
	To see this, let $z \in  L^p(\R^d)$ with $\norm{z}_p =1$ and suppose that $z = \l x + (1-\l) y$ for some $0< \l < 1$ and $x, y \in A$. Then,
	\[ 
	1 = \norm{z}_p = \norm{ \l x + (1-\l) y}_p \leq \l \norm{x}_p + (1-\l) \norm{y}_p \leq \l + (1-\l) = 1. 
	\]
	The only case for equality in Minkowski's equality $\norm{ \l x + (1-\l) y}_p \leq \l\norm{ x}_p +(1-\l)\norm{ y}_p$ is that $\l x = c (1-\l)y$ for some $c \geq 0$ and $\norm{x}_p=\norm{y}_p=1$. If $c=0,$ then $\l x =0,$ and $x=0$ contradicting $\norm{x}_p=1.$ Hence $c >0$ and $\l =c (1-\l)$ since $\norm{x}_p=\norm{y}_p=1.$ Then it follows that $x=y=z$ showing that $z$ is an extreme point of $A$. 
	
	On the other hand, if $0<\norm{z}_p<1,$ then $\norm{\frac{z}{\norm{z}_p}}_p=1 \leq 1$ and 
	\[ 
	z = \norm{z}_p \frac{z}{\norm{z}_p} + (1-\norm{z}_p) \, 0,
	\]
	where $ \norm{z}_p \in (0,1).$ If $\norm{z}_p=0,$ then
	\[ 
	z = 0 = \frac{1}{2} (+1) + \frac{1}{2} (-1).
	\]
	Hence, $z$ cannot be an extreme point of $A$ and the claim follows. Indeed, it is well-known that the weak closure of the unit sphere is the unit ball, that is,
	\[ 
	\overline{\varepsilon(A)} = \overline{ \{ z \in L^p(\R^d) \colon \norm{z}_p =1 \}} = A.
	\]
\end{example} 

Yet there is a stronger version of Strong Krein-Milman Theorem (Theorem~\ref{thm:strongKreinMilman}) which is known as Choquet theorem:

\begin{theorem}[Theorem 10.7, Simon~\cite{simon2011}]
	\label{thm:choquetthm}
	Let $A$ be a metrizable compact convex subset of $V.$ Then, every point in $A$ can be represented as a barycenter of a measure in $\M ( \eps(A) )$, that is, 
	\[ 
	r \Big( \M \big( \eps(A) \big) \Big) = A.
	\]
\end{theorem}

There is a partial converse to Theorem~\ref{thm:kreinmilman}$\colon$

\begin{theorem}[Theorem 9.4, Simon~\cite{simon2011}]
	\label{thm:milmanpartialconverse} 
	
	Let $A$ be a compact convex subset of $V.$ Let $B \subseteq A$ be such that 
	\[
	\cconv B = A.
	\]
	Then $\varepsilon(A) \subseteq \overline{B}.$
\end{theorem}

Theorems~\ref{thm:kreinmilman},~\ref{thm:barycenter},~\ref{thm:choquetthm} and~\ref{thm:milmanpartialconverse} establish the theoretical foundations to study compact convex sets. Inspired by Theorems~\ref{thm:kreinmilman} and~\ref{thm:choquetthm}, we investigate the image of the operator $r$ in detail in the next section.

\section{Choquet combinations}\label{sec:choquet}

Let $V$ be a 
Banach space equipped with the Borel $\sigma$-algebra $\B(V)$ and let $V^*$ be the dual of $V.$ Let $(\O,\F,\p)$ be a complete probability space. Notice that a convex combination $\sum_{i=1}^m \lambda_i a_i$ of elements of a subset $A$ of $V$ can be seen as the barycenter $\int_A a \, \mu(da)$ of $\mu = \sum_{i=1}^m \lambda_i \delta_{a_i}$ since for every $L \in V^* ,$ one has
\begin{align*}
L\bigg(\int_A a \, \mu(da)\bigg) &= L\bigg(\int_A a \, \sum_{i=1}^m \lambda_i \delta_{a_i}(da)\bigg) = \int_A L(a) \, \sum_{i=1}^m \lambda_i \delta_{a_i}(da) \\
&= \sum_{i=1}^m \lambda_i L(a_i) = L\bigg( \sum_{i=1}^m \lambda_i a_i \bigg).
\end{align*}
Generalizing convex combinations, we want to introduce a new terminology for barycenters with an emphasis on taking combinations:

\begin{definition} 
	\label{def:DeterministicChoquetCombination}
	Let $A$ be a subset of $V$. A point of the form $r(\mu) = \int_A a \, \mu(da)$ for some $\mu \in  \M(A)$ is called a \emph{Choquet combination} of points in $A$. We say that $A$ is a \emph{Choquet} set if it contains every Choquet combination of its elements, that is, if $\int_A a \, \mu(da) \in A$ for every $\mu \in  \M (A)$. 
\end{definition}

\begin{remark}\label{rem:ChoquetHull}
	Note that $V$ is a Choquet set. Moreover, arbitrary intersection of Choquet sets is a Choquet set. To see this, for an index set $J\neq \emptyset$, let $A_j$ be a Choquet set for every $j \in J$ and let $A \coloneqq \bigcap_{j \in J} A_j$. Let $\mu \in \M(A)$. Then, $\mu \in \M(B)$ for some Borel subset $B \subseteq A$. Since $B \subseteq A_j$, we have $\mu \in \M(A_j)$ for every $j \in J$. Since $A_j$ is a Choquet set, $r(\mu) \in A_j$ for every $j \in J$. Thus, $r(\mu) \in \bigcap_{j \in J} A_j = A$, showing that $A$ is a Choquet set. Hence, one can define the \emph{Choquet hull} of a set $A$ as the intersection of all Choquet sets containing $A$, which is equivalently the smallest Choquet set containing $A$; we denote the Choquet hull of $A$ by $\ch A$.
\end{remark}

\begin{remark}\label{rem:ChoquetConvexHullOperator}
	It is clear that $\ch$ is a hull operator, that is, for every $A,$ $B \subseteq V,$ one has the following properties:
	\begin{enumerate}[\bf (i)]
		\item \textbf{Extensive:} $A \subseteq \ch A$. 
		\item \textbf{Monotone:} If $A \subseteq B,$ then $\ch A  \subseteq \ch B.$
		\item \textbf{Idempotent:} $\ch(\ch A ) = \ch A .$ 
	\end{enumerate}
\end{remark}


\begin{remark}\label{rem:ChoquetisConvex}
	Since every Choquet set is convex, for every subset $A$ of $V,$ we have $\conv A \subseteq \ch A$. In infinite-dimensional spaces, one might have $\conv A \subsetneq  \ch A$; see the next example.
\end{remark}

\begin{example}\label{ex:ChoquetHullAreDifferent}
	Let $(\O,\F,\p) = ([0,1],\B([0,1]), \Leb )$, where $\B([0,1])$ is the Borel $\sigma$-algebra of the interval $[0,1]$ and $\Leb$ is the Lebesgue measure on $[0,1]$. Let $p \in (1,\infty)$, and let $V = L^p(\R)$. Consider 
	\[
	A = \{ \xi \in L^p(\R) \colon \xi \text{ is a simple function and } 0 \leq \xi \leq 1 \text{ almost surely} \}.
	\]
	The convexity of $A$ is clear. For each $n\in\N$, let $\xi_n = 1_{ \big [\frac{1}{2^n} , \frac{1}{2^{n-1}} \big)} \in A$, and define $\mu = \sum_{n=1}^{\infty} \frac{1}{2^n} \delta_{\xi_n}$. Our intuition suggests that $\mu \in \M (A) $ and 
	\begin{equation*}
	\int_A a \, \mu(da) = \sum_{n=1}^{\infty} \frac{1}{2^n} \xi_n = \sum_{n=1}^{\infty} \frac{1}{2^n}  1_{  \big [\frac{1}{2^n} , \frac{1}{2^{n-1}} \big)  }.
	\end{equation*}
	To verify this, let $L \in V^*$. Then, there exists $Z \in L^q(\R)$ such that $L(W) = \E[Z \, W]$ for every $W \in L^p(\R)$. We have 
	\begin{align*}
	L\bigg( \int_A a \, \mu(da) \bigg) &= \int_A L(a) \, \mu(da) = \sum_{n=1}^{\infty} \frac{1}{2^n} L(\xi_n) = \sum_{n=1}^{\infty} \frac{1}{2^n} \E \Big[Z 1_{ \big [\frac{1}{2^n} , \frac{1}{2^{n-1}} \big) } \Big ] \\
	&=  \E \bigg[Z \sum_{n=1}^{\infty} \frac{1}{2^n} 1_{ \big [\frac{1}{2^n} , \frac{1}{2^{n-1}} \big)  } \bigg ] = L\bigg( \sum_{n=1}^{\infty} \frac{1}{2^n}  1_{  \big [\frac{1}{2^n} , \frac{1}{2^{n-1}} \big) } \bigg).
	\end{align*}
	Taking countably many values, $\int_A a \, \mu(da)= \sum_{n=1}^{\infty} \frac{1}{2^n} 1_{  \big [\frac{1}{2^n} , \frac{1}{2^{n-1}} \big) }$ is not simple; hence $\int_A a \, \mu(da) \not \in A = \conv A$ although $\int_A a\, \mu(da)\in \ch A$ by definition.
\end{example} 

\begin{proposition}
	\label{prop:ChoquetHullisSameConvexHullinRd}
	The operators $\conv$ and $\ch$ coincide in finite-dimensional spaces. 
\end{proposition}

\begin{proof}
	Since we already have $\conv A \subseteq \ch A$ for every subset $A$ of $\R^d,$ it suffices to show that every convex set $A \subseteq \R^d$ is a Choquet set, which will be done by induction on $d$. 
	
	Initial step: Let $d=1$ and let $A \subseteq \R$ be convex. If $A$ is bounded, then, being convex, $A$ must be equal to one of the sets $(\inf A, \sup A),$ $[\inf A, \sup A),$ $(\inf A, \sup A],$ or $[\inf A, \sup A],$ each of which is a Choquet set. 
	The case where $A$ is unbounded follows similarly. 
	
	Induction step: Let $d \geq 2$ and by contradiction suppose that there exists a convex set $A \subseteq \R^d$ which is not a Choquet set. Then, there exists some $\mu \in \M(A)$ such that $\int_A a \, \mu(da) \not \in A.$ By the separating hyperplane theorem, there exists some nonzero vector $u \in \R^d$ such that 
	\[
	\sup_{y \in A} \ip{u,y} \leq  \ip{u,\int_A a \, \mu(da)} = \int_A \ip{u,a} \, \mu(da).
	\]
	Hence, we obtain
	\[
	\ip{u,a} = \sup_{y \in A} \ip{u,y} \quad \text{for } \mu \text{-almost every }a\in A.
	\]
Therefore, letting $t = \sup_{y \in A} \ip{u,y}$, 
 one can consider $\mu$ on the set $A \cap \{ x \in \R^d \colon \ip{u, x} = t\}$, living in a lower dimensional space. Since the set $A \cap \{ x \in \R^d \colon \ip{u, x} = t\}$ is convex, by the induction assumption, we must have $\int_A a \, \mu(da) \in A \cap \{ x \in \R^d \colon \ip{u, x} = t\} \subseteq A$, which is a contradiction. 
\end{proof}

Note that Theorem~\ref{thm:barycenter} reads as follows in our new setting:

\begin{theorem}
	\label{thm:barycenterChoquetTrans}
	Every compact convex subset $A$ of $V$ is a Choquet set.
\end{theorem}

Indeed, we show that compactness can be replaced with closedness in Theorem~\ref{thm:barycenterChoquetTrans}:

\begin{theorem}\label{thm:ConvexClosedThenChoquet}
	Every closed convex subset $A$ of $V$ is a Choquet set.  
\end{theorem}

\begin{proof}
	We argue by contradiction. Suppose that there is a closed convex set $A$ that is not a Choquet set. Then, there exists
	\[ 
	Y \coloneqq \int_A a \, \mu(da) \not \in A
	\]
	for some $\mu \in \M(A).$ Then $A$ and $\{Y \}$ are disjoint convex sets such that $\{Y\}$ is compact and $A$ is closed. By Hahn-Banach separation theorem, there exists a continuous linear functional $L \in V^*$ such that 
	\begin{align*}
	\sup_{a \in A}  L(a) &< \inf_{b \in \{ Y \} } L(b) = L(Y) =   L\bigg( \int_A a \, \mu(da) \bigg) \\
	&=  \int_A L(a) \, \mu(da)  \leq \int_A \, \sup_{b \in A}  L(b)\, \mu(da) = \sup_{b \in A}  L(b),
	\end{align*}
	which is a contradiction. 
\end{proof}

We have an immediate corollary of Theorem~\ref{thm:ConvexClosedThenChoquet} for the collection of $p$-integrable selections of a closed convex random set: 

\begin{corollary}\label{cor:LpXChoquet}
	Let $p \in \{0\} \cup [1,+\infty)$ and let $X$ be a closed convex random set. Then, $L^p(X)$ is a Choquet set. 
\end{corollary} 

\begin{proof}
	Let $X$ be a closed convex random set. Then, by Theorem~\ref{thm:selections}, $L^p(X)$ is closed and convex, hence is Choquet by Theorem~\ref{thm:ConvexClosedThenChoquet}.
\end{proof}

\begin{remark}
	\label{rem:ConvexChoquetRelation}
	For every subset $A$ of $V,$ since $\cconv A$ is a Choquet set containing $A$, we have $\conv A \subseteq \ch A  \subseteq \cconv A$. By taking the closure of the last two items, we obtain
	\[
	\conv A \subseteq \ch A  \subseteq \cconv A = \cch A.
	\]
\end{remark} 

It is well-known that the convex hull of a compact subset of $\R^d$ is compact. More generally, the closed convex hull of a compact subset of a Banach space is compact as well. In this case, the Choquet hull of a compact set coincides with its closed convex hull: 

\begin{proposition}[Theorem 3.28, Rudin~\cite{rudin1991}]
	\label{prop:ChoquetConvexHullCharacterization}
	Let $A$ be a compact subset of $V.$ Then,
	\[
	\ch A = \bigg \{ \int_A a \, \mu(da) \colon \mu \in \M(A) \bigg \} =\cconv A.
	\]
\end{proposition}

To obtain the convex hull of a set, it is enough to add all convex combinations of the points in the set. New points do not generate new convex combinations. The same principle holds for Choquet hull as well: once all Choquet combinations of the points in a set are added, one obtains a Choquet set. New points do not generate new Choquet combinations. This is shown by the next theorem which is the main result of this section:

\begin{theorem}
	\label{thm:ChoquetHullCharacterization}
	Let $A$ be subset of $V$. Then,
	\begin{equation}\label{choquethull:rhs}
	\ch A = \bigg \{ \int_A a \, \mu(da) \colon \mu \in \M(A) \bigg \}.
	\end{equation}
\end{theorem}

\begin{proof}
	Let $D$ denote the set on the right of \eqref{choquethull:rhs}. Since $\ch A$ is a Choquet set, we clearly have $\ch A \supseteq D$ since $\M(A) \subseteq \M(\ch A).$ It is also clear that $A \subseteq D.$ Hence, it suffices to show that $D$ is a Choquet set. Since the barycenter map $r \colon \M(A) \ni \mu \mapsto r(\mu) = \int_A a \, \mu(da)$ is measurable, the map $\phi \colon D \times \M(A) \to D \times D$ defined by
	\[
	\phi(b, \mu) \coloneqq (b, r(\mu))
	\]
	is jointly measurable. Consider the closed random set $X$ defined by $X(b) = \{b\}$ on the measurable space $(D,\B(D))$. 
	Then, the multifunction $S\colon D \to \mathcal{P}(\M (A))$ defined by 
	\begin{equation}\label{defnS}
	S(b) = r^{-1}(\{b\}) =  \{ \mu \in \M (A) \colon r(\mu)=b \} 
	\end{equation}
	is graph measurable since
	\[
	\gr(S) = \phi^{-1}(\gr(X)) \in \B(D) \otimes \B(\M(A)). 
	\]
	By definition $\dom(S)=D$, and by Theorem~\ref{thm:existenceofselections}, the multifunction $S$ admits a measurable selection $L\colon D \to \M(A).$ Then, for every $C \in  \B(V)$, the map $b \mapsto L(b) \mapsto L(b)(C)$ on $D$ is measurable being a composition of measurable maps.

	Hence, $(b,C)\mapsto L(b)(C)$ is a transition kernel from $(D,\B(D))$ into $(V,\B(V))$, where $\B(D)$ and $\B(V)$ denote the corresponding Borel $\sigma$-algebras. Let $v \in \M(D)$. We  show that $r(v)= \int_{D} b \, v(db) \in D$: 
	\begin{align*}
	\int_{D} b \, v(db) =& \int_{D} r\big(L(b)\big) \, v(db) = \int_{D} \int_A a \, L(b)(da) \, v(db) \\
	=& \int_{D} \int_A \Id(a) \,  L(b)(da) \, v(db)\\
	=& \int_{D} \big(L(\Id)\big)(b) \, v(db) = v(L(\Id))= (vL)(\Id)  \\
	=&  \int_A \Id(a) \, (vL)(da) =  \int_A a \, (vL)(da) \in D,
	\end{align*}
	where $\Id\colon A \to A $ is the identity mapping. Containing all Choquet combinations of its elements, $D$ is a Choquet set, concluding the proof. We refer the reader to Theorems~\ref{thm:transitionKernelsBochner} and~\ref{thm:distributionsBochner} in the appendix for the details of the above calculation.
\end{proof}

Unlike the argument in the proof of Theorem~\ref{thm:barycenter}, the barycenter map $r$ may not be continuous and the multifunction $S$ introduced above may not be closed in the non-compact case. We provide the following counterexample:

\begin{example}
	Let $V=\R$ and $A = [0,+\infty)$. Consider
	\[ 
	\mu_n = \frac{n-1}{n} \, \delta_{0} + \frac{1}{n} \, \delta_{n}
	\]
	for every $n \in \N$ and let $\mu = \delta_{0}$. Then, it is clear that $\mu_n \in \M(A)$ and $\mu \in \M(A)$ since $r(\mu_n) = 1$ and $r(\mu)=0.$ For every continuous bounded function $f \colon \R \to \R$, we have
	\[
	\int_{\R} f(x) \, \mu_n(dx) =\frac{n-1}{n} f(0) + \frac{1}{n} f(n) \xrightarrow{n} f(0) = \int_{\R} f(x) \, \mu(dx),\quad n\in\N.
	\]
	Hence, $\mu_n \rightarrow \mu$ weakly yet $r(\mu_n) = 1 \nrightarrow 0= r(\mu),$ showing that $r$ is not continuous and $S(1)$ is not closed, where $S$ is defined by \eqref{defnS}.
\end{example}

Theorem~\ref{thm:ChoquetHullCharacterization} provides an elegant quantitative characterization of the Choquet hull. We have the following partial result analogous to Proposition~\ref{prop:convdec=decconv}:

\begin{proposition}
	\label{prop:chisdec}
	For every $p \in \{0\} \cup [1,+\infty),$ the Choquet hull of a decomposable subset of $L^p(E)$ is decomposable. 
\end{proposition}

\begin{proof}
	Let $A$ be a decomposable subset of $L^p(E).$ Let $u,w \in \ch A$ and let $B \in \F$. We  show that $ u 1_B + w 1_{B^c}\in \ch A$ to conclude the decomposabilty of $\ch A$. Set $C=A.$ Since $u,w \in \ch A$, by Theorem~\ref{thm:ChoquetHullCharacterization}, we have $ u = \int_A a \, \mu(da) $ and $ w = \int_C c \, v(dc) $ for some $\mu, v \in \M(A)$. Consider $A \times C$ with the product measure $w = \mu \times v$ and by decomposability of $A$, consider the map $e \colon A \times C \to A$ defined by $e(a,c) = a 1_B + c 1_{B^c}$. Then, we have
	\begin{align*}
	u 1_B + w 1_{B^c}&= \int_A a \, \mu(da) \, 1_B + \int_C c \, v(dc) \, 1_{B^c} \\ 
	&= \int_C \int_A a \, \mu(da) \, v(dc) \, 1_B + \int_C \int_A c \, \mu(da) \, v(dc) \, 1_{B^c} \\
	&= \int_C \int_A a 1_B \, \mu(da) \, v(dc) + \int_C \int_A c 1_{B^c} \, \mu(da) \, v(dc)  \\
	&= \int_C \int_A a 1_B + c 1_{B^C}\, \mu(da) \, v(dc)  = \int_{A\times C} e(a,c) \, w(da \times dc) \\
	&= \int_{A\times C} \Id( e(a,c) ) \, w(da \times dc) = \int_A \Id(a) \, s(da) \in \ch A
	\end{align*}
	for $s = w \circ e^{-1}$, the distribution of $e$ under $w$, therefore showing the decomposability of $\ch A$. We refer the reader to Theorem~\ref{thm:distributionsBochner} in the appendix for the details of the above calculation.
\end{proof}

We continue by considering decompositions and explore their relation to barycenters: 

\section{Choquet decompositions} \label{sec:choquetdec}

Let $E$ be a separable Banach space equipped with the Borel $\sigma$-algebra $\B(E)$ and let $E^*$ be the dual of $E.$ Let $(\O,\F,\p)$ be a complete probability space. For every $p \in \{0\} \cup [1,+\infty)$, by Theorem~\ref{thm:selections} for every $\emptyset \not = A \subseteq L^p(E),$ $\clpdec(A)= L^p(F_A)$ for a unique closed random set $F_A$ up to almost sure equality, the independence of $p$ in the definition of $F_A$ will become apparent in Proposition~\ref{prop:SelectionsFor0andPchd} below. 


It is clear that the decomposition $\sum_{i=1}^m 1_{B_i} \xi_i$ of $(\xi_i)_{i=1}^m \subseteq A$ along a measurable partition $(B_i)_{i=1}^m$ is in $A$ whenever $A$ is decomposable. Notice that the coefficients $(1_{B_i})_{i=1}^m$ have the property that $1_{B_i} \geq 0$ for every $i\in \{1,\cdots,m\}$ and $\sum_{i=1}^m 1_{B_i} = 1.$ Hence they play the role of random convex coefficients, with the addition that for every $\o \in \O,$ only one coefficient is present. For every $\o \in \O,$ the expression $\sum_{i=1}^m 1_{B_i}(\o) \xi_i (\o)$ can be seen as the barycenter of the measure $\sum_{i=1}^m 1_{B_i}(\o) \delta_{\xi_i (\o)}.$ Combining them on $\O,$ we obtain a transition probability kernel $K$ defined by
\[
K(\o)(B) \coloneqq \sum_{i=1}^m 1_{B_i}(\o) \delta_{\xi_i (\o)}(B)= \delta_{\big( \sum_{i=1}^m 1_{B_i} \xi_i \big) (\o)}(B)
\]
from $(\O, \F)$ into $(E, \B (E))$. The barycenter of this kernel is given by
\[
\int_E x \, K(dx)= \int_E x \, \sum_{i=1}^m 1_{B_i}(\o) \delta_{\xi_i (\o)}(dx)  = \int_E x \, \delta_{\big( \sum_{i=1}^m 1_{B_i} \xi_i \big)}(dx) =  \sum_{i=1}^m 1_{B_i} \xi_i
\]
since, for every $L \in E^*$, we have
\[
L\bigg(\int_E x \, \delta_{\big( \sum_{i=1}^m 1_{B_i} \xi_i \big)}(dx)\bigg) = \int_E L(x) \, \delta_{\big( \sum_{i=1}^m 1_{B_i} \xi_i \big)}(dx) = L\bigg( \sum_{i=1}^m 1_{B_i} \xi_i \bigg).
\]

In this spirit, a transition probability kernel $K$ from $(\O, \F)$ to $(E, \B (E))$ is said to have a \emph{barycenter} in $L^p(E)$, where $p\in \{0\}\cup[1,\infty)$, if there exists a point $Y \in L^p(E)$ satisfying 
\[
L \circ Y = \int_E L(x) \, K(dx) 
\]
for every $L \in E^*.$ Since $E^*$ separates points in $E$, a transition probability kernel $K$ can have at most one barycenter, which is denoted by $r(K)$ or $\int_E x \, K(dx)$ whenever it exists.

Let $X \colon \O \to \mathcal{P}(E)$ be a graph measurable multifunction. Then $Z(\o, x) \coloneqq 1_{X(\o)}(x)$ is jointly measurable since 
\[
\{ (\o,x) \in \O \times E \colon Z(\o, x)=1 \}= \{ (\o,x) \in \O \times E \colon x \in X(\o) \} =\gr(X) \in \F \otimes \B(E)
\]
by the graph measurability of $X$. Then,
\[
\int_E Z(\o, x) \, K(\o,dx) = \int_E 1_{X(\o)}(x)\, K(\o,dx) = K(\o, X(\o))
\]
is measurable in $\o$. For every subset $A$ of $E,$ let $\Kp (A)$ denote the set of all regular Borel transition probability kernels $K$ with $K(\o, F_A(\o))=1$ for $\p$-almost every $\o\in\O$, having barycenters in $L^p(E)$. The barycenter of a transition probability kernel $K$ relates to the barycenters of the measures $K(\o), \o\in\O,$ in the following way:


\begin{remark}
	\label{rem:ChdDual}
	If a transition probability kernel $K$ from $(\O, \F)$ to $(E, \B (E))$ has a barycenter $Y$ in $L^0(E),$ then $Y(\o)$ is the barycenter of $K(\o)$ in $E$ for $\p$-almost every $\o\in\O$. 
\end{remark}


Let us also denote by $\K_p^\d(A)$ the set of all transition kernels $K\in \K_p(A)$ such that $K(\o)$ is a Dirac measure for $\p$-almost every $\o\in\O$. Analogous to the construction in Section \ref{sec:choquet}, generalizing decompositions, we want to introduce a new terminology for barycenters of transition probability kernels with an emphasis on taking decompositions:

\begin{definition} 
	\label{def:ChoquetDecomposition}
	Let $p \in \cb{0}\cup [1,+\infty)$. For every subset $A$ of $L^p(E),$ a point of the form $K \, \Id = \int_E x \, K(dx)$ for some $K \in  \Kp^\d(A)$ is called a \emph{$p$-Choquet decomposition} of points in $A$. We say that a set $A \subseteq L^p(E)$ is a \emph{$p$-Choquet decomposable} set if it contains every $p$-Choquet decomposition of its elements, that is, if $\int_E x \, K(dx) \in A$ for every $K \in  \Kp^\d(A)$.
\end{definition}

\begin{remark}
	\label{rem:ChoquetDecomposableHull}
	Note that $L^p(E)$ is a $p$-Choquet decomposable set. In addition, arbitrary intersection of $p$-Choquet decomposable sets is a $p$-Choquet decomposable set. Indeed, for an arbitrary index set $J$, let $A_j$ be a $p$-Choquet decomposable set for every $j \in J$ and let $A = \bigcap _{j \in J} A_j$. Let $K \in \Kp^\d(A)$. Then $K \in \Kp^\d(A_j)$ for every $j \in J$. Since $A_j$ is a $p$-Choquet decomposable set, $r(K) \in A_j$ for every $j \in J$. Thus, $r(K) \in \bigcap_{j \in J} A_j = A$, showing that $A$ is a $p$-Choquet decomposable set. Hence, one can define the \emph{$p$-Choquet decomposable hull} of a set $A$ as the intersection of all $p$-Choquet decomposable sets containing $A$, which is equivalently the smallest $p$-Choquet decomposable set containing $A$; we denote it by $\pchd A$. 
\end{remark}

\begin{remark}
	\label{rem:ChoquetDecomposableHullOperator}
	It is clear that $\pchd$ is a hull operator, that is, for every $A,B\subseteq L^p(E)$, we have the following properties:
	\begin{enumerate}[\bf (i)]
		\item \textbf{Extensive:} $A \subseteq \pchd A.$ 
		\item \textbf{Monotone:} If $A \subseteq B$, then $\pchd A \subseteq \pchd B.$
		\item \textbf{Idempotent:} $\pchd(\pchd A) = \pchd A .$
	\end{enumerate}
\end{remark}

Next, we prove a qualitative characterization of $p$-Choquet decomposability in terms of decomposability and strong closedness: 

\begin{proposition}
	\label{prop:pChoquetDecomposabilityChar}
	Let $A$ be a nonempty subset of $L^p(E).$ Then, $A$ is $p$-Choquet decomposable if and only if $A$ is decomposable and strongly closed.
\end{proposition}

\begin{proof}
	Suppose that $A$ is nonempty, decomposable and strongly closed and let $K \in \Kp^\d(A)$. Then, $K \, \Id (\o) \in F_A(\o)$ for $\p$-almost every $\o\in\O$; hence, $K \, \Id \in L^p(F_A) = \clpdec(A) = A$.
	
	Conversely, let $A$ be $p$-Choquet decomposable. For every $a \in \clpdec(A)=L^p(F_A)$ and $\o\in\O$, define $K(\o) \coloneqq \delta_{a(\o)}$. Then, $K \in \Kp^\d(A)$ and we have 
	\[
	(K \, \Id)(\o) = \int_E \Id(x) \, K(\o,dx) =\int_E x \, \delta_{a(\o)}(dx) = a(\o)
	\]
	for $\p$-almost every $\o\in\O$. Therefore, being a $p$-Choquet decomposition of elements in $A$, $a = K \, \Id$ belongs to $A$, showing that $A = \clpdec(A)$.
\end{proof}

Although $p$-Choquet decomposability does not seem to be a new concept, studying $p$-Choquet hull as an operator is insightful as we will observe in Theorems~\ref{thm:chpchd=pchcdh} and~\ref{thm:operatordiscussion} below. We have two immediate corollaries of Proposition~\ref{prop:pChoquetDecomposabilityChar}:

\begin{corollary}
	\label{cor:pchdHull}
	For every nonempty subset $A$ of $L^p(E),$ we have $\pchd A = \clpdec(A).$
\end{corollary}

\begin{corollary}
	\label{cor:LpXisChd}
	Let $A$ be a nonempty subset of $L^p(E).$ Then, $A$ is $p$-Choquet decomposable if and only if $A = L^p(X)$ for a unique closed random set $X$ in $E.$ 
\end{corollary}

\begin{remark}
	\label{rem:pchdisDecomposable}
	For every nonempty subset $A$ of $L^p(E)$, since $\pchd A = \clpdec(A)$ is a decomposable set containing $A$, we have $\dec A \subseteq \pchd A = \clpdec(A)$. Yet it is possible to have $\dec A \subsetneq  \pchd A$, see Example~\ref{example:pchdandDecDifferent} below.
\end{remark}

\begin{example}
	\label{example:pchdandDecDifferent}
	As in Example~\ref{ex:ChoquetHullAreDifferent}, let $(\O,\F,\p) = ((0,1),\B((0,1)), \Leb )$. Let $p \in (1,\infty),$ and $E = \R.$ Consider 
	\[
	A = \{ \xi \in L^p(\R) \colon \xi \text{ is a simple function and } 0 \leq \xi \leq 1 \text{ almost surely} \}.
	\]
	The decomposability of $A$ is clear, yet the infinite decomposition
	\[
	a = \sum \limits_{n=1 }^{\infty}  1_{\big [\frac{1}{2^n} , \frac{1}{2^{n-1}} \big)} \frac{1}{2^n} 
	\]
	of $(\frac{1}{2^n})_{n\in \N}$ along the partition $( [\frac{1}{2^n} , \frac{1}{2^{n-1}}))_{n\in \N}$
	is not simple. Hence $a \not \in A,$ showing that $A$ is not $p$-Choquet decomposable. More precisely, considering
	\[
	K \coloneqq \sum_{n=1}^{\infty} 1_{ \big [\frac{1}{2^n} , \frac{1}{2^{n-1}} \big) } \delta_{\frac{1}{2^n}},
	\]
	we have $K \in \Kp^\d(A)$. However, $r(K) \not\in A$. Indeed, we have
	\[
	\pchd A = \clpdec(A)= \{ \xi \in L^p(\R) \colon 0 \leq \xi \leq 1 \text{ almost surely} \} = L^p([0,1]). 
	\]
\end{example} 

One of the main results of this section is a relation between Choquet hull and $p$-Choquet decomposable hull analogous to Propositions~\ref{prop:convdec=decconv} and~\ref{prop:chisdec}:

\begin{theorem}
	\label{thm:chpchd=pchcdh}
	The $p$-Choquet decomposable hull of a nonempty convex subset of $L^p(E)$ is Choquet. 
	Moreover, for every $\emptyset \not = A \subseteq L^p(E),$ one has
	\begin{align*}
	\pchd \ch A &= \pchd \conv A = L^p(\cconv F_A) = \cconv L^p(F_A)\\ 
	&= \cconv \pchd A= \cconvdec A = \cdecconv A. 
	\end{align*}
\end{theorem}

\begin{proof}
	For the first part, let $A$ be a nonempty convex subset of $L^p(E).$ Then, $\pchd A = \clpdec(A)$ is convex and strongly closed by Proposition~\ref{prop:convdec=decconv}, hence is Choquet by Theorem~\ref{thm:ConvexClosedThenChoquet}. For the second part, let $A \subseteq L^p(E)$. Then, by Proposition~\ref{prop:convdec=decconv} and Theorem~\ref{thm:selections}, we have 
	\begin{align*}
	\pchd \conv A &= \clpdec \conv A = \cdecconv A = \cconvdec A \\
	&= \cconv \clpdec A = \cconv \pchd A= \cconv L^p(F_A) = L^p(\cconv F_A).  
	\end{align*}
	Since both $\pchd \ch A$ and $\pchd \conv A$ are Choquet and $p$-Choquet decomposable sets containing $A$, we also have $\pchd \ch A = \pchd \conv A$. 
	%
\end{proof}

Considering a subset $A$ of $L^p(E)$, there is an almost surely unique random closed set $X$ such that $\pchd A = L^p(X)$. However, $A$ can also be seen as a subset of $L^0(E)$. Then, there is an almost surely unique random closed set $Y$ such that $\ochd A = L^0(Y).$ It is natural to expect $X=Y$ almost surely, which is the case:
\begin{proposition}
	\label{prop:SelectionsFor0andPchd}
	Let $\emptyset \not = A \subseteq L^p(E)$ so that $\pchd A = \clpdec A = L^p(X)$ for an almost surely unique random closed set $X.$ Then,
	\[
	\ochd A = \clodec A = L^0(X).
	\]
\end{proposition}

\begin{proof}
	Let $\emptyset \not = A \subseteq L^p(E)$ so that $\pchd A = \clpdec A = L^p(X)$ for an almost surely unique random closed set $X.$ Since $\dec A \subseteq L^p(X) \subseteq L^0(X),$ we have $\clodec A \subseteq L^0(X).$ Conversely, let $Z \in A \subseteq L^p(X)$ be fixed. For every $Y \in L^0(X),$ consider
	\[
	Y^k =Y \, 1_{\{|Y| \leq k\}} + Z \, 1_{\{|Y|>k\}}
	\]
	for each $k \in \N$. Since $Y^k \in L^0(X) \cap L^p(E) = L^p(X) = \clpdec A$, there is a sequence $(X^k _n)_{n\in\N}$ in $\dec A$ such that $ X^k _n \xrightarrow{n} Y^k$ in $L^p$. Then, $ X^k _n \xrightarrow{n} Y^k$ in $L^0$ as well, hence $Y^k \in \clodec A$. 
	It remains to show that $Y^k \xrightarrow{k} Y$ in $L^0$ to conclude that $Y \in \clodec A$, showing $L^0(X) \subseteq \clodec A$. To that end, for every $\varepsilon >0$, we have
	$\p(|Y^k - Y| \geq \varepsilon ) \leq \p(|Y| > k ) \xrightarrow{k} 0$, showing that $ Y^k \xrightarrow{k} Y$ in $L^0$.
\end{proof}

We have an immediate corollary of Proposition~\ref{prop:SelectionsFor0andPchd} comparing $\pchd A$ and $\ochd A$ for $\emptyset \not = A \subseteq L^p(E)$: 
%
\begin{corollary}
	\label{cor:relatingPchdAndOchd}
	Let $\emptyset \not = A \subseteq L^p(E)$ so that $\pchd A = \clpdec A = L^p(X)$ for an almost surely unique random closed set $X.$ Then, we have
	\[
	\pchd A = \clpdec A = L^p(X) = L^0(X) \cap L^p(E) = \clodec A \cap L^p(E) = \ochd A \cap L^p(E) .
	\]
\end{corollary}

Combining Proposition~\ref{prop:SelectionsFor0andPchd} and Corollary~\ref{cor:relatingPchdAndOchd}, we get a comparison between $\pchd A$ and $\schd A$ for $1 \leq s < p$ as given by the next corollary.

\begin{corollary}
	\label{cor:relatingPchdAndSchd}
	Let $1\leq s < p<\infty$. Let $\emptyset \not = A \subseteq L^p(E)$ so that $\pchd A = \clpdec A = L^p(X)$ for an almost surely unique random closed set $X.$ Then,
	\[
	\schd A = \ochd A \cap L^s(E) = L^0(X) \cap L^s(E) = L^s(X).
	\]
	Hence,
	\[
	\schd A \cap L^p(E) = \ochd A \cap L^s(E) \cap L^p(E)  = \ochd A \cap L^p(E) = \pchd A .
	\]
\end{corollary}

Corollary~\ref{cor:relatingPchdAndSchd} suggests that the operator $\ochd$ is the essential one and other $\pchd$ operators can be obtained from the operator $\ochd$ via truncation. 

\begin{remark}
	\label{rem:ProjTypeDefChd}
	As stated in Corollary~\ref{cor:relatingPchdAndOchd}, for every $\emptyset \not = A \subseteq L^p(E),$ we have $\pchd A = \ochd A \cap L^p(E).$ Although the operator $\pchd$ is only defined for subsets of $L^p(E),$ $\ochd A \cap L^p(E)$ makes sense for every $A \subseteq L^0(E).$ Hence, we may define
	\[
	\pchd A\coloneqq \ochd A \cap L^p(E)
	\]
	for every $A \subseteq L^0(E)$, which coincides with the original definition when $A \subseteq L^p(E)$.

	Note that since $\ochd A = \clodec A = L^0(F_A)$ for an almost surely unique random closed set $F_A$, 
	\[
	\pchd A = \ochd A \cap L^p(E) = L^0(F_A) \cap L^p(E) = L^p(F_A)
	\]
	is a $p$-Choquet decomposable subset of $L^p(E)$.
\end{remark}

Theorem~\ref{thm:chpchd=pchcdh} holds with the extended definition in Remark~\ref{rem:ProjTypeDefChd}.

\begin{theorem}
	\label{thm:chpchd=pchcdhExt}
	The $p$-Choquet decomposable hull of a nonempty convex subset of $L^0(E)$ is Choquet. 
	Moreover, for every $A \subseteq L^0(E)$ such that $\pchd A \not = \emptyset,$ one has
	\[
	\pchd \ch A = \pchd \conv A = L^p(\cconv F_A) = \cconv L^p(F_A) = \cconv \pchd A.
	\]
\end{theorem}

\begin{proof}
	For the first part, let $A$ be a nonempty convex subset of $L^0(E).$ Then, by Theorem~\ref{thm:selections}, $\ochd A = \clodec A = L^0(F_A)$ for an almost surely unique closed convex random set $F_A$ by the convexity of $A$. Then, by Corollary~\ref{cor:LpXChoquet},
	\[
	\pchd A = \ochd A \cap L^p(E) = L^0(F_A) \cap L^p(E) = L^p(F_A)
	\]
	is Choquet. For the second part, let $A \subseteq L^0(E)$ such that $\pchd A \not = \emptyset$. Then, by Theorem~\ref{thm:chpchd=pchcdh}, we have 
	\[
	\ochd \ch A = \ochd \conv A = L^0(\cconv F_A) = \cconv L^0(F_A) = \cconv \ochd A .
	\]
	Hence, by Theorem~\ref{thm:selections}, we have
	\begin{align*}
	\pchd \ch A &= \ochd \ch A \cap L^p(E) =\ochd \conv A \cap L^p(E) = \pchd \conv A \\
	&= L^0(\cconv F_A)\cap L^p(E)  =  L^p(\cconv F_A) =\cconv L^p(F_A)= \cconv \pchd A,
	\end{align*}
	which completes the proof.
	%
\end{proof}

The extension of Proposition~\ref{prop:SelectionsFor0andPchd} for subsets of $L^0(E)$ follows easily:

\begin{proposition}
	\label{prop:SelectionsFor0andPchdExt}
	Let $\emptyset \not = A \subseteq L^0(E)$ so that $\pchd A = L^p(X)$ for an almost surely unique random closed set $X$. Then,
	\[
	\ochd A = \clodec A = L^0(X).
	\]
\end{proposition}

\begin{proof}
	 Let $\emptyset \not = A \subseteq L^0(E)$ so that $\pchd A = L^p(X)$ for an almost surely unique random closed set $X$. We have $\ochd A = L^0(F_A)$ for an almost surely unique random closed set $F_A.$ Then, $\pchd A = L^p(F_A)$ and suppose $\pchd A = L^p(X)$ for a random closed set $X$. Hence, $L^p(F_A) = L^p(X)$, and by Theorem~\ref{thm:selections}, $F_A = X$ $\p$-almost surely. Therefore, 
	\[
	\ochd A = \clodec A = L^0(F_A) = L^0(X),
	\]
	as desired.
\end{proof}

Corollaries~\ref{cor:relatingPchdAndOchd} and~\ref{cor:relatingPchdAndSchd} extend easily as well: 

\begin{corollary}
	\label{cor:relatingPchdAndOchdExt}
	For $\emptyset \not = A \subseteq L^0(E)$ with $\emptyset \not = \pchd A = L^p(X)$ for an almost surely unique random closed set $X$. Then,
	\[
	\pchd A = L^p(X) = L^0(X) \cap L^p(E) = \clodec A \cap L^p(E) = \ochd A \cap L^p(E) .
	\]
\end{corollary}

\begin{corollary}
	\label{cor:relatingPchdAndSchdExt}
	Let $1\leq s < p<\infty$. Let $\emptyset \not = A \subseteq L^0(E)$ with $\pchd A \not = \emptyset$ so that $\pchd A = \clpdec A = L^p(X)$ for an almost surely unique random closed set $X$. Then,
	\[
	\schd A = \ochd A \cap L^s(E) = L^0(X) \cap L^s(E) = L^s(X).
	\]
	Hence,
	\[
	\schd A \cap L^p(E) = \ochd A \cap L^s(E) \cap L^p(E)  = \ochd A \cap L^p(E) = \pchd A .
	\]
\end{corollary}

{
}

Inspired by Corollaries~\ref{cor:relatingPchdAndOchdExt} and~\ref{cor:relatingPchdAndSchdExt}, let $\calLpv$ denote the collection of subsets of $\Lpv$ and let $\chdLpv$ denote the collection of $p$-Choquet decomposable subsets of $\Lpv .$ Then the operator $\pchd \colon \calLov \to \calLov ,$ taking values in $\chdLpv,$ is still monotone and idempotent yet is not extensive anymore, hence $\pchd$ is not a hull operator on $\calLov.$ Notice that $\ochd$ is still a hull operator on $\calLov.$ In the next theorem, we investigate how the operators $\pchd,$ $\schd$ and $\ochd$ interact with each other when applied in succession, which is another main result from this section. 

\begin{theorem}
	\label{thm:operatordiscussion}
	Let $\emptyset \not = A \subseteq L^0(E)$ and let $1 \leq s \leq p.$ 
	\begin{enumerate}
		\item \label{prop:operatordiscussion1} If $\pchd A \not = \emptyset,$ then $\ochd \pchd A = \ochd A.$
		
		\item \label{prop:operatordiscussion2} $\pchd \ochd A = \pchd A.$
		
		\item \label{prop:operatordiscussion3} $\pchd \pchd A = \pchd A.$
		
		\item \label{prop:operatordiscussion4} $\pchd \schd A = \pchd A.$
		
		\item \label{prop:operatordiscussion5} If $\pchd A \not = \emptyset,$ then $\schd \pchd A = \schd A.$
	\end{enumerate}
\end{theorem}

\begin{proof}
		1. We have $\pchd A = \ochd A \cap L^p(E) \subseteq \ochd A.$ Since $\ochd$ is a hull operator, we get
		\[
		\ochd \pchd A \subseteq \ochd \ochd A = \ochd A.
		\]
		Conversely, let $b \in \pchd A$ be fixed. For each $a \in A,$ and for each $n \in \mathbb{N},$ we have 
		\[ 
		a_n = a 1_{ \{ n-1 \leq |a| < n \} } + b 1_{ \{ n-1 \leq |a| < n \}^c } \in \ochd A \cap L^p(E) = \pchd A.
		\]
		Then,
		\[ 
		a= \sum_{n=1}^{\infty} a_n 1_{ \{ n-1 \leq |a| < n \} } \in \ochd \pchd A.
		\]
		Hence, $A \subseteq \ochd \pchd A$. Since $\ochd$ is a hull operator, we have 
		\[
		\ochd A \subseteq \ochd \ochd \pchd A =\ochd \pchd A.
		\]
		2. Since $\ochd$ is a hull operator, we have 
		\[
		\pchd \ochd A = \ochd \ochd A \cap L^p(E) = \ochd A \cap L^p(E) = \pchd A.
		\]
		3. We already had this when we observed that $\pchd A$ is a $p$-Choquet decomposable set and the extended definition coincides with the original definition which acts as identity on $p$-Choquet decomposable sets. 
		Yet we provide another proof purely based on simple manipulations: If $\pchd A = \emptyset$, then the claim is trivial. Consider the case $\pchd A \not = \emptyset$. Then, by \ref{prop:operatordiscussion1}., we have
		\[
		\pchd \pchd A = \ochd \pchd A \cap L^p(E) = \ochd A \cap L^p(E)  = \pchd A.
		\]
		4. If $\schd A = \emptyset,$ then $\pchd = \emptyset$ as well, making the claim trivial. Consider the case $\schd A \not = \emptyset.$ Then, by \ref{prop:operatordiscussion1}., we have
		\[ 
		\pchd \schd A = \ochd \schd A \cap L^p(E) = \ochd A \cap L^p(E) = \pchd A.
		\]
		We provide another proof for the case $\schd A \not = \emptyset$: by \ref{prop:operatordiscussion1}. and \ref{prop:operatordiscussion2}., we have
		\[ 
		\pchd \schd A = \pchd \ochd \schd A  = \pchd \ochd A  = \pchd A.
		\]	
		5. By \ref{prop:operatordiscussion1}., we have 
		\[
		\schd \pchd A = \ochd \pchd A \cap L^s(E) =  \ochd A \cap L^s(E) =\schd A .
		\]
		We provide another proof for the case $\pchd A \not = \emptyset \colon$ by \ref{prop:operatordiscussion1}. and \ref{prop:operatordiscussion2}., we have
		\[ 
		\schd \pchd A = \schd \ochd \pchd A  = \schd \ochd A  = \schd A.
		\]	
\end{proof}

\begin{remark}
	The assumption $\pchd A \not = \emptyset$ in Theorem~\ref{thm:operatordiscussion}~(\ref{prop:operatordiscussion1}. and~\ref{prop:operatordiscussion5}.) is worth some attention: It is possible to have a set $A \subseteq L^0(E)$ such that $\pchd A= \clodec A \cap L^p(E) \not = \emptyset$ yet $\dec A \cap L^p(E) = \emptyset$. Let $(\O,\F,\p) = ((0,1),\B((0,1)), \Leb)$ and let $E=\R$. For each $k\in\N$, consider
	\[
	X_k = \sum \limits_{n=1 }^{\infty}  1_{ \big [\frac{1}{2^n} , \frac{1}{2^{n-1}} \big) } 2^{n-k}.
	\] 
	It is clear that $X_k \not \in L^1(\R)$. Let $A \coloneqq \dec \{ X_k \colon k \in \N\}$. Then, $A \cap L^1(\R) = \emptyset$. However, $0 \in \cl_0 A \cap L^1(\R)$ since, for every $\eps >0,$ we have $\p(|X_k| \geq \varepsilon ) \xrightarrow{k} 0$.
\end{remark}

We continue by considering random convex combinations and explore their relation to barycenters.

\section{Choquet convex decompositions}\label{sec:choquetconvdec}
Let $E$ be a separable Banach space equipped with the Borel $\sigma$-algebra $\B(E)$ and let $E^*$ be the dual of $E$. Let $(\O,\F,\p)$ be a complete probability space. 

For measurable functions $\l_i$ with $\l_i \geq 0$ for every $i \in \{1,\ldots,m\}$ and $\sum_{i=1}^m \l_i = 1$ $\p$-almost surely, the expression $\sum_{i=1}^m \l_i \xi_i$ is called a random convex combination of $(\xi_i)_{i=1}^m$. Notice that for $\p$-almost every $\o\in\O$, the expression $\sum_{i=1}^m \l_i (\o) \xi_i (\o)$ can be seen as the barycenter of the measure $\sum_{i=1}^m \l_i (\o) \delta_{\xi_i (\o)}$. Combining them on $\O$, we obtain a transition kernel $K$ defined by 
\[
K(\o)(B) \coloneqq \sum_{i=1}^m \l_i (\o) \delta_{\xi_i (\o)}(B)
\]
from $(\O, \F)$ into $(E, \B (E))$. The random barycenter of $K$ is given by
\[
\int_E x \, K(dx)= \int_E x \, \sum_{i=1}^m \l_i(\o) \delta_{\xi_i (\o)}(dx) = \sum_{i=1}^m \l_i(\o)  \int_E x \, \delta_{\xi_i (\o)}(dx) =  \sum_{i=1}^m \l_i \xi_i
\]
since, for each $L \in E^*$, we have 
\begin{align*}
L\bigg(\int_E x \, \sum_{i=1}^m \l_i(\o) \delta_{\xi_i (\o)}(dx)\bigg) &= \int_E L(x) \, \sum_{i=1}^m \l_i(\o) \delta_{\xi_i (\o)}(dx) = \sum_{i=1}^m \l_i(\o) \int_E L(x) \, \delta_{\xi_i (\o)}(dx) \\
& = \sum_{i=1}^m \l_i(\o) L(\xi_i (\o)) =  L\bigg( \sum_{i=1}^m \l_i \xi_i \bigg). 
\end{align*}

Generalizing the convex combinations and the Choquet decompositions in Sections \ref{sec:choquet} and \ref{sec:choquetdec} , we introduce a new terminology for barycenters of transition probability kernels with an emphasis on taking convex decompositions:

\begin{definition} 
	\label{def:ChoquetConvexDecomposition}
	Let $p \in \cb{0}\cup [1,+\infty)$ and let $A$ be a subset of $L^p(E)$. A point of the form $K \, \Id = \int_E x \, K(dx)$ for some $K \in  \Kp(A)$ is called a \emph{$p$-Choquet convex decomposition} of points in $A$. We say that $A$ is \emph{$p$-Choquet convex decomposable} if it contains every $p$-Choquet convex decomposition of its elements, that is, if $\int_E x \, K(dx) \in A$ for every $K \in  \Kp(A)$.
\end{definition}

\begin{remark}
	\label{rem:ChoquetConvexDecomposableHull}
	Note that $L^p(E)$ is a $p$-Choquet convex decomposable set. Moreover, arbitrary intersection of $p$-Choquet convex decomposable sets is a $p$-Choquet convex decomposable set. To show this, let $A_j$ be a $p$-Choquet convex decomposable set for every $j \in J$ and let $A = \bigcap _{j \in J} A_j$, where $J$ is an arbitrary index set. Let $K \in \Kp(A)$. Then, $K \in \Kp(A_j)$ for every $j \in J$. Since $A_j$ is a $p$-Choquet convex decomposable set, $r(K) \in A_j$ for every $j \in J$. Thus, $r(K) \in \bigcap_{j \in J} A_j = A$, showing that $A$ is a $p$-Choquet convex decomposable set. Hence, one can define the \emph{$p$-Choquet convex decomposable hull} of a set $A$ as the intersection of all $p$-Choquet convex decomposable sets containing $A$, which is equivalently the smallest $p$-Choquet convex decomposable set containing $A$; we denote it by $\pchcd A$. 
\end{remark}

\begin{remark}
	\label{rem:ChoquetConvexDecomposableHullOperator}
	It is clear that $\pchcd$ is a hull operator, that is, for every $A,B\subseteq L^p(E)$, we have the following properties:
	\begin{enumerate}[\bf (i)]
		\item \textbf{Extensive:} $A \subseteq \pchcd A.$ 
		\item \textbf{Monotone:} If $A \subseteq B$ then $\pchcd A \subseteq \pchcd B.$
		\item \textbf{Idempotent:} $\pchcd(\pchcd A) = \pchcd A .$
	\end{enumerate}
\end{remark}

\noindent
We have a quantitative characterization of $p$-Choquet convex decomposability in terms of the collection of $p$-integrable selections of a closed convex random set:

\begin{proposition}
	\label{prop:LpXisChcd}
	Let $A$ be a nonempty subset of $L^p(E)$. Then, $A$ is $p$-Choquet convex decomposable if and only if $A = L^p(X)$ for an almost surely unique closed convex random set $X$ in $E$.
\end{proposition}

\begin{proof}
	Let $X$ be a closed convex random set and let $K \in \K _p(L^p(X)).$ Then $K(\o)$ is supported on $X(\o)$ for $\p$-almost every $\o\in\O$. Since $X(\o)$ is closed and convex by Theorem~\ref{thm:ConvexClosedThenChoquet}, $X(\o)$ is Choquet for $\p$-almost every $\o\in\O$. Then, we have
	\[
	(K \, \Id)(\o) = \int_E \Id(x) \, K(\o,dx)= \int_E x \, K(\o,dx) \in \ch X(\o)=X(\o)
	\]
	for $\p$-almost every $\o\in\O$. Hence, $K \, \Id \in L^p(X)$, showing that $L^p(X)$ is $p$-Choquet convex decomposable. 
	
	Conversely, suppose that $A$ is $p$-Choquet convex decomposable. Since $A$ is $p$-Choquet decomposable, by Corollary~\ref{cor:LpXisChd}, $A = L^p(X)$ for some random closed set $X$. By Theorem~\ref{thm:selections}, the convexity of $A$ ensures the convexity of $X.$ 
\end{proof}

Although Proposition~\ref{prop:LpXisChcd} looks similar to Corollary~\ref{cor:LpXisChd}, our proof for Proposition~\ref{prop:LpXisChcd} is rather indirect since no transition probability kernel is defined explicitly. We have three immediate corollaries of Proposition~\ref{prop:LpXisChcd}: 

\begin{corollary}
	\label{cor:pchcdiffconvexdecclosed}
	Let $A$ be a nonempty subset of $L^p(E).$ Then, $A$ is $p$-Choquet convex decomposable if and only if $A$ is convex, decomposable and strongly closed.
\end{corollary}

Although $p$-Choquet convex decomposability does not seem to be a new concept, studying $p$-Choquet decomposable hull as an operator is insightful, we refer the reader to Corollary~\ref{cor:pchcdHull} and Theorem~\ref{thm:operatordiscussionConvex}. We have two immediate corollaries of Proposition~\ref{prop:pChoquetDecomposabilityChar}:
\begin{corollary}
	\label{cor:pchcdHull}
	Let $\emptyset \not = A \subseteq L^p(E)$. Then, one has 
	\begin{align*}
	\pchcd A &= \pchd \ch A = \pchd \conv A = L^p(\cconv F_A) = \cconv L^p(F_A) \\
	&= \cconv \pchd A= \cconvdec A = \cdecconv A
	\end{align*}
where $F_A$ is the almost surely unique random set provided in Proposition~\ref{prop:SelectionsFor0andPchd}.
\end{corollary}

\begin{proof}
	Let $\emptyset \not = A \subseteq L^p(E)$. By Corollary~\ref{cor:LpXisChd}, $\pchd A = \clpdec A = L^p(F_A)$ for an almost surely unique random closed set $F_A$. Then, by Theorem~\ref{thm:selections} and Proposition~\ref{prop:LpXisChcd}, we have
	\[
	\cconv \pchd A = \cconv L^p(F_A) = L^p(\cconv F_A) \subseteq \pchcd A.
	\]
	By Proposition~\ref{prop:LpXisChcd}, $L^p(\cconv F_A)$ is a $p$-Choquet convex decomposable set containing $A$, hence we have $L^p(\cconv F_A) \subseteq \pchcd A$ as well. The rest follows from Theorem~\ref{thm:chpchd=pchcdh}.
\end{proof}

Corollary~\ref{cor:pchcdHull} factors the operator $\pchcd$ as $\pchd \ch$ and $\cconv \pchd$, which is useful for computational purposes since one does not have to work with arbitrary transition kernels in $\K_p(A).$ We consider a particular case of Corollary~\ref{cor:pchcdHull} for closed random sets:

\begin{corollary}\label{prop:pchcdLpX}
	Let $X$ be a closed random set with $L^p(X) \not = \emptyset$. Then, 
	\[
	\pchcd L^p(X) = \cconv \pchd L^p(X) = \cconv L^p(X) = L^p(\cconv X).
	\]
\end{corollary}

Considering a subset $A$ of $L^p(E)$, there is an almost surely unique random closed set $X$ such that $\pchd A = L^p(X)$ and $\pchcd A = L^p(\cconv X)$. However, $A$ can also be seen as a subset of $L^0(E)$. Then, there is an almost surely unique random closed convex set $Y$ such that $\ochcd A = L^0(Y)$. It is natural to expect $\cconv X=Y$, which easily follows from Proposition~\ref{prop:SelectionsFor0andPchd} and Corollary~\ref{cor:pchcdHull}:

\begin{proposition}
	\label{prop:SelectionsFor0andPchcd}
	Let $\emptyset \not = A \subseteq L^p(E)$ so that $\pchd A = \clpdec A = L^p(X)$ for an almost surely unique random closed set $X$. Then,
	\[
	\ochcd A = \cconv \ochd A = \cconv L^0(X) = L^0(\cconv X).
	\]
\end{proposition}

\begin{proof}
	Let $\emptyset \not = A \subseteq L^p(E)$ so that $\pchd A = \clpdec A = L^p(X)$ for an almost surely unique random closed set $X$. By Proposition~\ref{prop:SelectionsFor0andPchd}, we have
	\[
	\ochd A = \clodec A = L^0(F_A).
	\]
	Then, by Corollary~\ref{cor:pchcdHull}, we have
	\begin{align*}
	\ochcd A &= \cconv \ochd A = \cconv L^0(F_A) = L^0(\cconv F_A),
	\end{align*}
	as desired.
\end{proof}

We have an immediate corollary of Proposition~\ref{prop:SelectionsFor0andPchcd} comparing $\pchcd A$ and $\ochcd A$ for $\emptyset \not = A \subseteq L^p(E) \colon$
\begin{corollary}
	\label{cor:relatingPchcdAndOchcd}
	Let $\emptyset \not = A \subseteq L^p(E)$ so that $\pchd A = \clpdec A = L^p(X)$ for an almost surely unique random closed set $X$. Then, 
	\begin{align*}
	\pchcd A &= \cconv \pchd A = \cconv L^p(X)=L^p(\cconv X) = L^0(\cconv X) \cap L^p(E) 
	= \ochcd  A \cap L^p(E) .
	\end{align*}
\end{corollary}

By combining Proposition~\ref{prop:SelectionsFor0andPchcd} and Corollary~\ref{cor:relatingPchcdAndOchcd}, we get a comparison between $\pchd A$ and $\schd A$ for $1 \leq s < p$ as described in the next corollary.

\begin{corollary}
	\label{cor:relatingPchcdAndSchcd}
	Let $\emptyset \not = A \subseteq L^p(E)$ so that $\pchd A = \clpdec A = L^p(X)$ for an almost surely unique random closed set $X$. Then,
	\[
	\schcd A = \ochcd A \cap L^s(E) = \cconv \schd A = \cconv L^s(X) = L^s(\cconv X).
	\]
	Hence,
	\[
	\schcd A \cap L^p(E) = \ochcd A \cap L^s(E) \cap L^p(E)  = \ochcd A \cap L^p(E) = \pchcd A.
	\]
\end{corollary}

Corollary~\ref{cor:relatingPchcdAndSchcd} suggests that the operator $\ochcd$ is the essential one and other $\pchcd$ operators can be obtained from the operator $\ochcd$ via truncation. 

\begin{remark}
	\label{rem:ProjTypeDefChcd}
	As stated in Corollary~\ref{cor:relatingPchcdAndOchcd}, for every $\emptyset \not = A \subseteq L^p(E),$ we have $\pchcd A = \ochcd A \cap L^p(E).$ Although the operator $\pchcd$ is only defined for subsets of $L^p(E),$ $\ochcd A \cap L^p(E)$ makes sense for every $A \subseteq L^0(E).$ Hence, we can define 
	\[
	\pchcd A \coloneqq \ochcd A \cap L^p(E)
	\]
	for every $A \subseteq L^0(E)$, which coincides with the original definition when $A \subseteq L^p(E)$. 

	Note that since $\ochcd A = \clodec \conv A = L^0(\cconv F_A)$ for an almost surely unique random closed set $F_A$, the set 
	\[
	\pchcd A = \ochcd A \cap L^p(E) = L^0(\cconv F_A) \cap L^p(E) = L^p(\cconv F_A)
	\]
	is a $p$-Choquet convex decomposable subset of $L^p(E).$
\end{remark}

Corollary~\ref{cor:pchcdHull} holds with the extended definition:

\begin{corollary}
	\label{cor:pchcdHullExt}
	For every subset $A \subseteq L^0(E)$ such that $\pchd A \not = \emptyset$, one has
	\begin{align*}
	\pchcd A &= \pchd \ch A = \pchd \conv A = L^p(\cconv F_A) = \cconv L^p(F_A)= \cconv \pchd A .
	\end{align*}
\end{corollary}

\begin{proof}
	Let $A \subseteq L^0(E)$ such that $\pchd A \not = \emptyset$. Then, by Corollary~\ref{thm:chpchd=pchcdhExt}, we have
	\[
	\pchd \ch A = \pchd \conv A = L^p(\cconv F_A) = \cconv L^p(F_A) = \cconv \pchd A.
	\]
	By Corollary~\ref{cor:pchcdHull}, we also have
	\[
	\ochcd A = \ochd \ch A = \ochd \conv A = L^0(\cconv F_A) = \cconv L^0(F_A) = \cconv \ochd A.
	\]
	Intersecting with $L^p(E),$ we get
	\begin{align*}
	\pchcd A &= \ochcd A \cap L^p(E)= \ochd \ch A \cap L^p(E) = \pchd \ch A = \pchcd \conv A \\
	&= L^0(\cconv F_A) \cap L^p(E)= L^p(\cconv F_A)= \cconv L^p(F_A)= \cconv \pchd A .
	\end{align*}
\end{proof}

The extension of Proposition~\ref{prop:SelectionsFor0andPchcd} follows easily: 

\begin{proposition}
	\label{prop:SelectionsFor0andPchcdExt}
	Let $\emptyset \not = A \subseteq L^0(E)$ with $\emptyset \not = \pchd A = \clpdec A = L^p(X)$ for a unique random closed set $X$. Then,
	\[
	\ochcd A = \cconv \ochd A = \cconv L^0(X) = L^0(\cconv X).
	\]
\end{proposition}

\begin{proof}
	Let $\emptyset \not = A \subseteq L^0(E)$ with $\emptyset \not = \pchd A = \clpdec A = L^p(X)$ for a unique random closed set $X$. By Proposition~\ref{prop:SelectionsFor0andPchdExt}, we have 
	\[
	\ochd A = \clodec A = L^0(X).
	\]
	By Corollary~\ref{cor:pchcdHullExt}, we have
	\begin{align*}
	\ochcd A = \cconv \ochd A  = \cconv L^0(F_A)= L^0(\cconv F_A) .
	\end{align*}
\end{proof}

The extensions of Corollaries~\ref{cor:relatingPchcdAndOchcd} and~\ref{cor:relatingPchcdAndSchcd} follow easily as well: 
\begin{corollary}
	\label{cor:relatingPchcdAndOchcdExt}
	Let $\emptyset \not = A \subseteq L^0(E)$ with $\emptyset \not = \pchd A = \clpdec A = L^p(X)$ for a unique random closed set $X$. Then we have 
	\begin{align*}
	\pchcd A &= \cconv \pchd A = \cconv L^p(X)=L^p(\cconv X) \\&= L^0(\cconv X) \cap L^p(E) 
	= \ochcd  A \cap L^p(E) .
	\end{align*}
\end{corollary}

\begin{corollary}
	\label{cor:relatingPchcdAndSchcdExt}
	Let $\emptyset \not = A \subseteq L^0(E)$ with $\emptyset \not = \pchd A = \clpdec A = L^p(X)$ for a unique random closed set $X$. Then,
	\[
	\schcd A = \ochcd A \cap L^s(E) = \cconv \schd A = \cconv L^s(X) = L^s(\cconv X).
	\]
	Hence,
	\[
	\schcd A \cap L^p(E) = \ochcd A \cap L^s(E) \cap L^p(E)  = \ochcd A \cap L^p(E) = \pchcd A.
	\]
\end{corollary}

Inspired by Corollaries~\ref{cor:relatingPchcdAndOchcdExt} and~\ref{cor:relatingPchcdAndSchcdExt}, let $\chcdLpv$ denote the collection of all $p$-Choquet convex decomposable subsets of $\Lpv$. Then, the operator $\pchcd \colon \calLov \to \calLov $, taking values in $\chcdLpv$, is still monotone and idempotent yet is not extensive anymore, hence $\pchcd$ is not a hull operator on $\calLov$. Notice that $\ochcd$ is still a hull operator on $\calLov$. We now investigate how the operators $\pchcd$, $\schcd$ and $\ochcd$ interact with each other when applied in succession:

\begin{theorem}
	\label{thm:operatordiscussionConvex}
	Let $\emptyset \not = A \subseteq L^0(E)$ and let $1 \leq s \leq p.$ 
	\begin{enumerate}
		\item If $\pchcd A \not = \emptyset,$ then $\ochcd \pchcd A = \ochcd A.$
		
		\item $\pchcd \ochcd A = \pchcd A.$
		
		\item $\pchcd \pchcd A = \pchcd A.$
		
		\item $\pchcd \schcd A = \pchcd A.$
		
		\item If $\pchcd A \not = \emptyset,$ then $\schcd \pchcd A = \schcd A.$
	\end{enumerate}
\end{theorem}

\begin{proof}
		1. By Theorem~\ref{thm:operatordiscussion}~(\ref{prop:operatordiscussion1}.) and Corollary~\ref{cor:pchcdHullExt}, we have 
		\[
		\ochcd \pchcd A = \ochd \conv \pchd \conv A =  \ochd \pchd \conv A = \ochd \conv A  = \ochcd A.
		\]
		2. By Theorem~\ref{thm:operatordiscussion}~(\ref{prop:operatordiscussion2}.) and Corollary~\ref{cor:pchcdHullExt}, we have 
		\[
		\pchcd \ochcd A = \pchd \conv \ochd \conv A = \pchd \ochd \conv A = \pchd \conv A =\pchcd A.
		\]
		3. We already had this when we observed that $\pchcd A$ is a $p$-Choquet convex decomposable set and the extended definition coincides with the original definition which acts as identity on $p$-Choquet convex decomposable sets. 
		Yet we provide with another proof purely based on simple manipulations: By Theorem~\ref{thm:operatordiscussion}~(\ref{prop:operatordiscussion3}.) and Corollary~\ref{cor:pchcdHullExt} we have
		\[
		\pchcd \pchcd A = \pchd \conv \pchd \conv A = \pchd \pchd \conv A = \pchd \conv A =  \pchcd A.
		\]		
		4. By Theorem~\ref{thm:operatordiscussion}~(\ref{prop:operatordiscussion4}.) and Corollary~\ref{cor:pchcdHullExt}, we have 
		\[
		\pchcd \schcd A = \pchd \conv \schd \conv A = \pchd \schd \conv A = \pchd \conv A = \pchcd A.
		\]		
		5. By Theorem~\ref{thm:operatordiscussion}~(\ref{prop:operatordiscussion5}.) and Corollary~\ref{cor:pchcdHullExt}, we have 
		\[
		\schcd \pchcd A = \schd \conv \pchd \conv A = \schd \pchd \conv A = \schd \conv A = \schcd A .
		\]
\end{proof}

\section{Conclusion}\label{sec:conc}

In this paper, we introduce Choquet combinations and Choquet hulls generalizing convex combinations and convex hulls. We show that Choquet and convex hulls coincide in finite-dimensional spaces, yet generally differ in infinite-dimensional spaces, Choquet hulls being the larger ones. We show that closed convex sets are Choquet and provide a quantitative characterization for Choquet hulls. By generalizing finite decompositions, we introduce Choquet decompositions and Choquet decomposable hulls. Choquet decomposability is characterized by being strongly closed and decomposable. By generalizing random convex combinations, we introduce Choquet convex decompositions and Choquet convex decomposable hulls. We show that the Choquet convex decomposable hull operator is the composition of the Choquet decomposable hull operator with the convex hull operator. Lastly, Choquet decomposable and Choquet convex decomposable hulls are investigated as hull operators in relation to the collection of measurable selections of random sets. 

\section{Future directions}\label{sec:fut}

Let $E$ be a separable Banach space equipped with its Borel $\sigma$-algebra $\B(E)$ and let $(\O,\F,\p)$ be a probability space. 
Let us restate Krein-Milman Theorem~\ref{thm:kreinmilman} for $L^p(E)\colon$

\begin{theorem}[Theorem 8.14, Simon~\cite{simon2011}]
\label{thm:kreinmilmanRestated}
	Let $A$ be a (weakly or strongly) compact convex subset of $L^p(E)$. Then, $\eps(A) \not =\emptyset$ is a Baire $G_{\delta}$ set in the sense that $\eps(A) $ is a countable intersection of open sets and
	\begin{equation}\label{equ:kreinMilman}
	A = \cconv \eps(A). 
	\end{equation}
\end{theorem}


We continue by the relation between the decomposability of a convex set and of its extreme points:

\begin{proposition}
	\label{prop:decomposabilityOfExtremePts}
	Let $A$ be a convex subset of $L^p(E)$. If $A$ is decomposable, then $\eps (A)$ is decomposable. When $A$ is also compact, then $A$ is decomposable if and only if $\eps (A)$ is decomposable. 
\end{proposition}

\begin{proof}
Let $A$ be a convex subset of $L^p(E)$. For the first part suppose that $A$ is decomposable and let $\xi,$ $\zeta \in \eps (A)$ and $B \in \F.$ Then, $1_B \xi+ 1_{B^c} \zeta \in A$ by the decomposability of $A.$ Suppose that $\xi 1_B + \zeta 1_{B^c} = \l y + (1-\l) z$ with $y,$ $z \in A$ and $\l \in (0,1)$. Then $\xi 1_B  = \l y 1_B+ (1-\l) z 1_B $ and $ \zeta 1_{B^c} = \l y 1_{B^c}+ (1-\l) z 1_{B^c}. $ Hence,
\[
\xi = \xi 1_B + \xi 1_{B^c} = \l y 1_B+ (1-\l) z 1_B + \xi 1_{B^c} = \l (y 1_B + \xi 1_{B^c})+ (1-\l) (z 1_B + \xi 1_{B^c}),
\]  
where $y 1_B + \xi 1_{B^c} \in A$ and $z 1_B + \xi 1_{B^c} \in A$ by the decomposability of $A.$ By the extremity of $\xi,$ we have $\xi = y 1_B + \xi 1_{B^c}$ and $\xi = z 1_B + \xi 1_{B^c}.$ Hence $\xi 1_B = y 1_B $ and $\xi 1_B = z 1_B$. Similarly, we have $\zeta 1_{B^c} = y 1_{B^c} $ and $\zeta 1_{B^c} = z 1_{B^c}$. Hence, we have $\xi 1_B + \zeta 1_{B^c} = y =z$ showing the extremity of $\xi 1_B + \zeta 1_{B^c}$.

For the second part, suppose that $A$ is also compact and $\eps (A)$ is decomposable. Then, by Krein-Milman Theorem~\ref{thm:kreinmilmanRestated}, we have $A = \cconv \eps(A)$. Then, $A$ is decomposable by Proposition~\ref{prop:convdec=decconv}.
\end{proof}

By taking the closed decomposable hull of both sides in \eqref{equ:kreinMilman}, one obtains the following corollary:

\begin{corollary}
\label{cor:kreinmilmanExtended}
Let $A$ be a (weakly or strongly) compact convex subset of $L^p(E)$. Then,
\begin{equation}\label{equ:kreinmilmanExtended}
	\cconvdec \eps (A) = \overline{\dec }A.
\end{equation}
\end{corollary}

\begin{proof}
Let $A$ be a compact convex subset of $L^p(E)$. By Krein-Milman Theorem~\ref{thm:kreinmilmanRestated} we have $A = \cconv \eps(A).$ Taking the closed decomposable hull of both sides by Proposition~\ref{prop:convdec=decconv} we have
\[
\overline{\dec }A =\overline{\dec} \cconv \eps(A)= \cdecconv \eps (A) = \cconvdec \eps(A).
\]
\end{proof}

\begin{remark}
\label{rem:comparison}
When $A$ is (weakly or strongly) compact convex and decomposable, $\eps(A)$ is decomposable by Proposition~\ref{prop:decomposabilityOfExtremePts}. Then the equation \eqref{equ:kreinmilmanExtended} becomes the same as \eqref{equ:kreinMilman}. 

When $A$ is compact convex and $\overline{\dec }A$ is compact, then by Theorem~\ref{thm:kreinmilmanRestated} and Corollary~\ref{cor:kreinmilmanExtended} we have 
\begin{equation}\label{equ:kreinMilmanComp}
	\cconv \eps(\overline{\dec }A) = \overline{\dec }A 
	\end{equation}
	and
\begin{equation}\label{equ:kreinmilmanExtendedComp}
	\cconv \dec \eps (A) = \overline{\dec }A.
\end{equation}
\end{remark}

Let us restate Theorem~\ref{thm:milmanpartialconverse} for $L^p(E)\colon$

\begin{theorem}[Theorem 9.4, Simon~\cite{simon2011}]
\label{thm:MilmanRestated}
	Let $A$ be a (weakly or strongly) compact convex subset of $L^p(E)$. Let $B \subseteq A$ be such that $\cconv B = A.$ Then 
	\begin{equation}\label{equ:Milman}
	\eps(A) \subseteq \overline{B}. 
	\end{equation}
\end{theorem}

\begin{remark}
When $A$ is (weakly or strongly) compact convex and decomposable, by equations \eqref{equ:kreinMilmanComp} and \eqref{equ:kreinmilmanExtendedComp}, and Milman's Theorem~\ref{thm:MilmanRestated} we have 
\begin{equation}\label{equ:extremeComparison}
	\eps(\overline{\dec }A) \subseteq \overline{ \dec} \eps (A) .
\end{equation}
However, it is possible to have $\eps(\overline{\dec }A) \not \subseteq  \dec  \eps (A) $ and $\eps(\overline{\dec }A) \not \supseteq \eps (A)$ as shown in the next two examples. Yet one might prefer $\dec \eps (A)$ to $\eps(\overline{\dec }A),$ hence Corollary~\ref{cor:kreinmilmanExtended} to Theorem~\ref{thm:kreinmilmanRestated} for computational purposes. 
\end{remark}

\begin{example}
\label{example:comparison1}
Let $(\O,\F,\p) = ([0,1],\B([0,1]), \Leb ).$ Let $\mathbb{S}^1$ denote the unit circle and $\mathbb{D}^2$ denote the unit disk in $\R^2.$ Let $f_s = s $ for every $s \in \mathbb{S}^1$ and let $A = \conv \{ f_s \colon s \in \mathbb{S}^1 \}.$ Then it is clear that $\eps (A) =  \{ f_s \colon s \in \mathbb{S}^1 \}$ and $\dec \eps (A) =  \{ g \colon \norm{g}_2 = 1 \text{ a.s with } g \text{ being simple} \}.$ On the other hand, we have $\dec (A) =  \{ g \colon \norm{g}_2 \leq  1 \text{ a.s with } g \text{ being simple} \}$, $\overline{ \dec } (A) =  \{ g \colon \norm{g}_2 \leq  1 \text{ a.s with } g \text{ being measurable} \}$ and $\eps \dec (A) =  \{ g \colon \norm{g}_2 = 1 \text{ a.s with } g \text{ being measurable} \}.$ Hence, $\eps(\overline{\dec }A) \not \subseteq  \dec  \eps (A) $. 
\end{example}

\begin{example}
\label{example:comparison2}
Let $(\O,\F,\p) = ([0,1],\B([0,1]), \Leb )$ and consider $f_1 = 0 \, 1_{[0,\frac{1}{2})} + 5 \, 1_{[\frac{1}{2},1]},$ $f_2 = 1$ and $f_3 = -1$. Let $A = \conv \{f_1,f_2,f_3\}.$ It is easy to see that $A$ is  weakly compact, $\eps(A) = \{ f_1,f_2,f_3\}$, 
\begin{align*}
\dec A &= \{ g 1_{[0,\frac{1}{2})} + h 1_{[\frac{1}{2},1]} \colon -1 \leq g \leq 1 \text{ and } -1 \leq h \leq 5 \text{ with } g,  h \text{ being simple} \},  
\end{align*}
and 
\begin{align*}
\overline{\dec} A &= \{ g 1_{[0,\frac{1}{2})} + h 1_{[\frac{1}{2},1]} \colon -1 \leq g \leq 1 \text{ and } -1 \leq h \leq 5 \text{ with } g,  h \text{ being measurable} \}. 
\end{align*}
Note that $\overline{\dec} A$ is weakly compact too. We have $f_1 \in \eps(A),$ yet $f_1 \not \in \eps(\overline{\dec }A) $ showing that $\eps(\overline{\dec }A) \not \supseteq \eps (A)\colon$ consider $u = f_2 \,  1_{[0,\frac{1}{2})} + f_1 \, 1_{[\frac{1}{2},1]} =  1 \, 1_{[0,\frac{1}{2})} +  5 \, 1_{[\frac{1}{2},1]} \in \dec A $ and $v = f_3 \, 1_{[0,\frac{1}{2})} + f_1 \, 1_{[\frac{1}{2},1]} = -1 \, 1_{[0,\frac{1}{2})} + 5 \, 1_{[\frac{1}{2},1]} \in \dec A.$ Then it is obvious that $f_1 = \frac{1}{2} u + \frac{1}{2} v,$ yet $f_1 \neq u,$ hence $f_1$ cannot be an extreme point of $\overline{\dec} A.$ 
\end{example}

Let us restate Choquet Theorem~\ref{thm:strongKreinMilman} for $L^p(E)$:

\begin{theorem}[Theorem 10.7, Simon~\cite{simon2011}]
	\label{thm:choquetthmRestated}
	Let $A$ be a metrizable compact convex subset of $L^p(E).$ Then
	\[ 
	\ch \eps(A) = A.
	\]
\end{theorem}

As further directions, we would like to perform an analysis for Theorem~\ref{thm:choquetthmRestated} similar to Corollary~\ref{cor:kreinmilmanExtended} and Remark~\ref{rem:comparison}. Moreover, we would like to prove Krein-Milman/Choquet-type theorems for operators $\pchd$ and $\pchcd$ with potentially other notions of extreme points.

\section{Appendix}
\label{sec:Bochner}
We refer the reader to \c{C}{\i}nlar~\cite{cinlar2011} for the details regarding transition probability kernels and distributions, and to Mikusi\'{n}ski~\cite{Mikusinski1978} for the construction of Bochner integral, which is similar to that of Lebesgue integral.

\subsection{Bochner integrals and transition kernels}
\label{sebsec:bochnerAndKernels}


Let $(\A,\mathfrak{a})$ and $(\B,\mathfrak{b})$ be measurable spaces. A mapping $K \colon \A \times \mathfrak{b} \to \R^+$ is called a \emph{transition (probability) kernel} if 
\begin{itemize}
	\item Measurable part: for every $D \in \mathfrak{b}$, the mapping $a \mapsto K(a,D)$ is $\mathfrak{a}$-measurable,
	\item Measure part: for every $a \in \A$, the mapping $D \mapsto K(a,D)$ is a (probability) measure on $(\B,\mathfrak{b}).$
\end{itemize}

\begin{theorem}[Theorem 6.3, \c{C}{\i}nlar~\cite{cinlar2011}]
	\label{thm:transitionKernels}
	Let $K$ be a transition probability kernel from $(\A,\mathfrak{a})$ to $(\B,\mathfrak{b})$.
	\begin{enumerate}
		\item For a measurable function $f \colon \B \to \R_+$, the function $Kf \colon \A \to \R_+$ defined by
		\[
		Kf \, (a) \coloneqq \int_{\B} f(b) \, K(a,db)
		\]
		is a random variable.
		
		\item For a probability measure $v$ on $(\A,\mathfrak{a})$, the set function $vK:\mathfrak{b}\to \R_+$ defined by
		\[
		vK \, (D) \coloneqq \int_{\A} K(a,D) \, v(da) 
		\]
		is a probability measure on $(\B,\mathfrak{b})$.
		
		\item For a measurable function $f \colon \B \to \R_+$ and a probability measure $v$ on $(\A,\mathfrak{a})$, we have
		\[
		v \, (Kf) = (vK) \, f,
		\]
		that is,
		\[
		v \, (Kf) = \int_{\A} Kf \,(a) \, v(da) =  \int_{\B} f(b) \, vK(db)  = (vK) \, f.
		\]
	\end{enumerate}
\end{theorem}

Bochner integrals enjoy similar properties:
\begin{theorem}
	\label{thm:transitionKernelsBochner}
	Let $(\C, \B(\C))$ be a Banach space equipped with its Borel $\sigma$-algebra $\B(\C)$. Let $K$ be a transition probability kernel from $(\A,\mathfrak{a})$ to $(\B,\mathfrak{b})$.
	\begin{enumerate}
		\item For a measurable function $f \colon \B \to \C$, the function $Kf \colon \A \to \C$ defined by
		\[
		Kf \, (a) \coloneqq \int_{\B} f(b) \, K(a,db) 
		\]
		is measurable provided that the above Bochner integrals exist.
		
		\item 
		For a probability measure $v$ on $(\A,\mathfrak{a})$, the set function $vK\colon \mathfrak{b}\to \R_+$ defined by
		\[
		vK \, (D) \coloneqq \int_{\A} K(a,D) \, v(da) 
		\]
		is a probability measure on $(\B,\mathfrak{b})$. 
		
		\item For a measurable function $f \colon \B \to \C$ and a probability measure $v$ on $(\A,\mathfrak{a})$,
		\[
		v \, (Kf) = (vK) \, f,
		\]
		that is, 
		\[
		v \, (Kf) = \int_{\A} Kf \,(a) \, v(da) =  \int_{\B} f(b) \, vK(db)  = (vK) \, f
		\]
		provided that the above Bochner integrals exist. 
	\end{enumerate}
\end{theorem}

\begin{proof}
	Proof of the first part, which is avoided, is quite standard relying on a monotone class argument. The second part is the same as of Theorem~\ref{thm:transitionKernels}. Hence we only focus on the last part: Assuming that the above Bochner integrals exist, let $Z(a) \coloneqq Kf \,(a) = \int_{\B} f(b) \, K(a,db)$ for each $a\in \A$ in Bochner sense, where $Z \colon \A \to \C.$ To show the equality of the two Bochner integrals 
	\[
	\int_{\A} Z(a) \, v(da) =  \int_{\B} f(b) \, vK(db),
	\]
	let $W \colon \C \to \R$ be a continuous linear functional. Then, by Theorem~\ref{thm:transitionKernels}, we have
	\begin{align*}
	W\bigg( \int_{\A} Z(a) \, v(da) \bigg) &\negthinspace=\negthinspace \int_{\A} W\big(Z(a)\big) \, v(da) 
	= \int_{\A} W\Big(\int_{\B} f(b) \, K(a,db)\Big) \, v(da) \\
	&\negthinspace=\negthinspace \int_{\A} \int_{\B} W\big(f(b)\big) \, K(a,db) \, v(da) 
	= \int_{\B} W\big(f(b)\big) \, vK(db) 
	\negthinspace=\negthinspace W\bigg(\int_{\B} f(b) \, vK(db) \bigg),
	\end{align*}
	where $W \circ f \colon \B \to \R$ is a real valued random variable. Hence, $v \, (K (W\circ f)) = (vK) \, (W\circ f)$. 
\end{proof}

\subsection{Bochner integrals and distributions}
\label{subsec:bochnerAndDist}
Let $(\A,\mathfrak{a},w)$ be a probability space and let $(\B,\mathfrak{b})$ be a measurable space.

\begin{theorem}[Theorem 2.4, \c{C}{\i}nlar~\cite{cinlar2011}]
	\label{thm:distributions}
	Let $Y\colon \A \to \B$ be a random variable and let $s$ be the distribution of $Y$ under $w.$ Then for every measurable function $f \colon \B \to \R^+,$ we have
	\[
	\int_{\A} f\circ Y (a) \, w(da) = \int_{\B} f(b) \, s(db). 
	\]
\end{theorem}

Bochner integrals enjoy similar properties:

\begin{theorem}
	\label{thm:distributionsBochner}
	Let $(\C, \B(\C))$ be a Banach space equipped with its Borel $\sigma$-algebra $\B(\C)$. Then, for every measurable function $f \colon \B \to \C$, we have
	\begin{equation}\label{eq:bochner}
	\int_{\A} f\circ Y (a) \, w(da) = \int_{\B} f(b) \, s(db)
	\end{equation}
	provided that the above Bochner integrals exist.
\end{theorem}

\begin{proof}
	Assuming that the above Bochner integrals exist, to show \eqref{eq:bochner}, let $W \colon \C \to \R$ be a continuous linear functional. Then, by Theorem~\ref{thm:distributions}, we have 
	\begin{align*}
	W\bigg( \int_{\A} f\circ Y (a) \, w(da) \bigg) &= \int_{\A} W\big(f\circ Y (a) \big) \, w(da) 
	= \int_{\A} (W \circ f) \circ Y (a) \, w(da) \\
	&= \int_{\B} (W \circ f) (b) \, s(db) 
	= \int_{\B} W \big( f(b)\big) \, s(db) 
	= W\bigg(\int_{\B} f(b) \, s(db) \bigg),
	\end{align*}
	where $W \circ f \colon \B \to \R$ is a real-valued random variable.
\end{proof}

\section*{Acknowledgments}\label{Acknowledgements}

The second author is grateful to Prof.\@ Ali Süleyman Üstünel for helpful conversations at the early stages of this research. The second author is grateful to T\"{U}B{\.I}TAK, The Scientic and Technological Research Council of Turkey, for their graduate scholarship 2210. 

\bibliography{ACreflist}
\bibliographystyle{AIMS.bst}

\end{document}